\pdfoutput=1
\RequirePackage{ifpdf}
\ifpdf % We~are running pdfTeX in pdf mode
\documentclass[pdftex]{sigma}
\else
\documentclass{sigma}
\fi

\numberwithin{equation}{section}

\newtheorem{Theorem}{Theorem}[section]
\newtheorem{Corollary}[Theorem]{Corollary}
\newtheorem{Lemma}[Theorem]{Lemma}
\newtheorem{Proposition}[Theorem]{Proposition}
 { \theoremstyle{definition}
\newtheorem{Definition}[Theorem]{Definition}
\newtheorem{Example}[Theorem]{Example}
\newtheorem{Remark}[Theorem]{Remark} }

\newcommand{\RR}{\mathbb{R}}

\newcommand{\ZZ}{\mathbb{Z}}
\newcommand{\CC}{\mathbb{C}}

\begin{document}
%\allowdisplaybreaks

\renewcommand{\PaperNumber}{011}

\FirstPageHeading

\ShortArticleName{Invariant Dirac Operators, Harmonic Spinors, and Vanishing Theorems in CR Geometry}

\ArticleName{Invariant Dirac Operators, Harmonic Spinors,\\ and Vanishing Theorems in CR Geometry}

\Author{Felipe LEITNER}

\AuthorNameForHeading{F.~Leitner}

\Address{Universit\"at Greifswald, Institut f\"ur Mathematik und Informatik,\\
Walter-Rathenau-Str.~47, D-17489 Greifswald, Germany}
\Email{\href{mailto:felipe.leitner@uni-greifswald.de}{felipe.leitner@uni-greifswald.de}}
%\URLaddress{\url{https://math-inf.uni-greifswald.de/institut/ueber-uns/mitarbeitende/leitner/}}

\ArticleDates{Received July 23, 2020, in final form January 22, 2021; Published online February 04, 2021}

\Abstract{We study Kohn--Dirac operators $D_\theta$ on strictly pseudoconvex CR manifolds with ${\rm spin}^{\mathbb C}$ structure of weight $\ell\in{\mathbb Z}$. Certain components of $D_\theta$ are CR invariants. We also derive CR invariant twistor operators of weight~$\ell$. Harmonic spinors correspond to cohomology classes of some twisted Kohn--Rossi complex. Applying a Schr\"odinger--Lichnerowicz-type formula, we prove vanishing theorems for harmonic spinors and (twisted) Kohn--Rossi groups. We also derive obstructions to positive Webster curvature.}


\Classification{32V05; 53C27; 58J50; 32L20}

\section{Introduction} \label{Sec1}

The classical Schr\"odinger--Lichnerowicz formula
\[
D^2 = \Delta +\frac{{\rm scal}}{4}
\]
of Riemannian geometry relates the square of the Dirac operator to the spinor Laplacian and scalar curvature.
This Weitzenb\"ock formula can be used to prove {\em vanishing theorems} for harmonic spinors on closed manifolds.
Via Hodge theory and Dolbeault's theorem this give rise to vanishing theorems for holomorphic cohomology
on K\"ahler manifolds (see~\cite{Hitch}). Moreover, via the {\em index theorem} for elliptic differential operators, the $\hat{\mathcal{A}}$-genus is understood to be an obstruction to positive scalar curvature on spin manifolds (see~\cite{Lich}).

Due to J.J.~Kohn there is also an {\em harmonic theory} for the Kohn--Laplacian on strictly pseudoconvex CR manifolds (see~\cite{FK,Koh}). Even though the Kohn--Laplacian is not elliptic, this theory shows that classes in the cohomology groups of the {\em tangential Cauchy--Riemann complex} (or {\em Kohn--Rossi complex}) are represented by harmonic forms. In particular, the (non-extremal) {\em Kohn--Rossi groups} are finite dimensional over closed manifolds.

In \cite{Tan} Tanaka describes this harmonic theory for the Kohn--Laplacian on $(p,q)$-forms with values in some CR vector bundle $E$
over (abstract) strictly pseudoconvex CR manifolds.
The Kohn--Laplacian is defined with respect to some pseudo-Hermitian structure $\theta$ and the corresponding {\em canonical connection}.
In particular, Tanaka derives Weitzenb\"ock formulas and proves vanishing theorems
for the Kohn--Rossi groups. On the other hand, in~\cite{Pet} R.~Petit introduces {\em spinor calculus} and Dirac-type operators to strictly pseudoconvex CR manifolds
with adapted pseudo-Hermitian structure (cf.~also \cite{LeiN1, Stadt}). Deriving some Schr\"odinger--Lichnerowicz-type formula
for the {\em Kohn--Dirac operator}, this approach gives rise to vanishing theorems for harmonic spinors over closed CR manifolds (cf.\ also~\cite{Kor}).

We study in this paper the Kohn--Dirac operator $D_\theta$ for ${\rm spin}^\CC$ structures of weight \mbox{$\ell\in\ZZ$}
on strictly pseudoconvex CR manifolds with adapted pseudo-Hermitian
structure~$\theta$. Our construction of~$D_\theta$ uses the Webster--Tanaka spinor derivative, only.
The Kohn--Dirac operator~$D_\theta$ does not behave naturally with respect to {\em conformal changes}
of the underlying pseudo-Hermitian structure. However, similar as in K\"ahler geometry, the spinor bundle~$\Sigma$ decomposes with respect to the CR structure
into eigenbundles~$\Sigma^{\mu_q}$
of certain eigenvalues~$\mu_q$. For $\mu_q=-\ell$ the restriction~$\mathcal{D}_\ell$ of the Kohn--Dirac operator to~$\Gamma\big(\Sigma^{\mu_q}\big)$
acts CR-covariantly.
This observation gives rise to CR invariants for the underlying strictly pseudoconvex CR manifold.

Complementary to $D_\theta$ we also have {\em twistor operators}.
In the spin case \cite{LeiN1} we discuss special solutions of the corresponding twistor equation, which realize
some lower bound for the square of the first non-zero eigenvalue of the Kohn--Dirac operator $D_\theta$.
For $\mu_q=\ell$ the corresponding twistor operator $\mathcal{P}_\ell$ is a CR invariant.

Analyzing the Clifford multiplication on the spinor bundle for ${\rm spin}^\CC$ structures over strictly pseudoconvex CR manifolds
shows that the Kohn--Dirac operator is a square root of the Kohn--Laplacian acting on $(0,q)$-forms with values in some CR line bundle $E$.
Thus, our discussion of the Kohn--Dirac operator fits well to Kohn's harmonic theory, as described in \cite{Tan}.
In particular, harmonic spinors correspond to cohomology classes of certain twisted Kohn--Rossi complexes. Computing the curvature term of
the corresponding Schr\"odinger--Lichnerowicz-type formula gives rise to vanishing theorems for twisted Kohn--Rossi groups.

For example,
on a closed, strictly pseudoconvex CR manifold $M$ of even CR dimension $m\geq 2$ with spin structure given by a square root $\sqrt{\mathcal{K}}$
of the canonical line bundle,
we have for $\mu_q=\ell=0$ the Schr\"odinger--Lichnerowicz-type formula
\[
\mathcal{D}_0^*\mathcal{D}_0 = \Delta^{\rm tr} + \frac{{\rm scal}^{\rm W}}{4}
\]
for the CR-covariant component $\mathcal{D}_0$ of $D_\theta$,
where $\Delta^{\rm tr}$ denotes the {\em spinor sub-Laplacian} and ${\rm scal}^{\rm W}$ is the Webster scalar curvature.
In this case harmonic spinors correspond to cohomology classes in the
Kohn--Rossi group $H^{\frac{m}{2}}\big(M,\sqrt{\mathcal{K}}\big)$.
Positive Webster scalar curvature ${\rm scal}^{\rm W}>0$ on $M$ immediately implies that this Kohn--Rossi group is trivial.
On the other hand, \mbox{$H^{\frac{m}{2}}\big(M,\sqrt{\mathcal{K}}\big)\neq\{0\}$} poses an obstruction to the existence of any
adapted pseudo-Hermitian structure~$\theta$ on~$M$ of positive Webster scalar curvature.
In this case the {\em Yamabe invariant} in~\cite{JL} for the given CR structure is non-positive.

In Sections~\ref{Sec2} to~\ref{Sec5} we introduce CR manifolds and pseudo-Hermitian geometry with ${\rm spin}^\CC$ structures.
In Section \ref{Sec6} the Kohn--Dirac and twistor operators are constructed. The CR-covariant
components $\mathcal{D}_\ell$ and $\mathcal{P}_\ell$ are determined in Section~\ref{Sec7}. In Section~\ref{Sec8} we recall
the Schr\"odinger--Lichnerowicz-type formula and derive a basic vanishing theorem for harmonic spi\-nors (see Proposition~\ref{vani}).
Section~\ref{Sec9} briefly reviews the harmonic theory for the Kohn--Laplacian. In Section~\ref{Sec10} we derive
 vanishing theorems for twisted Kohn--Rossi groups. In Section~\ref{Sec11} we discuss CR circle bundles of K\"ahler manifolds and
relate holomorphic cohomology groups to Kohn--Rossi groups. Finally, in Section~\ref{Sec12} we construct closed, strictly pseudoconvex CR mani\-folds
with $H^{\frac{m}{2}}\big(M,\sqrt{\mathcal{K}}\big)\neq 0$.

\section{Strictly pseudoconvex CR structures} \label{Sec2}

Let $M^n$ be a connected and orientable, real $C^\infty$-manifold of odd dimension $n=2m+1\geq 3$, equipped with a pair $(H(M),J)$
of a corank $1$ subbundle $H(M)$ of the tangent bundle $T(M)$ and a bundle endomorphism $J\colon H(M)\to H(M)$
with $J^2(X)=-X$ for any $X\in H(M)$. The Lie bracket $[\cdot,\cdot]$ of vector fields defines a bilinear skew pairing
\begin{align*}\{\cdot,\cdot\}\colon \ H(M)\times H(M) &\to T(M)/H(M),\\
 (X,Y)&\mapsto -[X,Y]\ \operatorname{mod} H(M),
\end{align*}
with values in the real line bundle $T(M)/H(M)$.

We call the pair $(H(M),J)$ a {\em strictly pseudoconvex CR structure} on $M$ (of hypersurface type and CR dimension $m\geq 1$)
if the following conditions are satisfied:
\begin{itemize}\itemsep=0pt
\item $\{JX,Y\}+\{X,JY\}=0$ $\operatorname{mod} H(M)$ for any $X,Y\in H(M)$ and
\item the symmetric pairing $\{\cdot,J\cdot\}$ on $H(M)$ is definite, i.e., $\{X,JX\}\neq 0$ for any $X\neq 0$,
\item the {\em Nijenhuis tensor} $N_J(X,Y)= [X,Y]-[JX,JY]+J([JX,Y]+[X,JY])$ vanishes identically for any $X,Y\in H(M)$.
\end{itemize}
Throughout this paper we will deal with strictly pseudoconvex CR structures on $M$.
For example, in the generic case when the Levi form is non-degenerate,
the smooth boundary of a {\em domain of holomorphy} in $\CC^{m+1}$ is strictly pseudoconvex.

The complex structure $J$ extends $\CC$-linearly to $H(M)\otimes\CC$, the complexification of the {\em Levi distribution}, and induces a decomposition
\[
H(M)\otimes\CC = T_{10}\oplus T_{01}
\]
into $\pm {\rm i}$-eigenbundles. Then a complex-valued $p$-form $\eta$ on $M$ is said to be of type $(p,0)$ if $\iota_Z\eta=0$ for all $Z\in T_{01}$.
This gives rise to the complex vector bundle $\Lambda^{p,0}(M)$ of $(p,0)$-forms on~$M$.
For the $(m+1)$st exterior power $\Lambda^{m+1,0}(M)$ of $\Lambda^{1,0}(M)$ we write $\mathcal{K}=\mathcal{K}(M)$.
This is the {\em canonical line bundle} of the CR manifold $M$ with first Chern class $c_1(\mathcal{K})\in H^2(M,\ZZ)$.
Its dual is the {\em anticanonical bundle}, denoted by~$\mathcal{K}^{-1}$.

When dealing with a strictly pseudoconvex CR manifold, we will often assume the existence and choice
of some $(m+2)$nd root $\mathcal{E}(1)$ of the anticanonical bundle $\mathcal{K}^{-1}$,
that is a complex line bundle over $M$ with
\[
\mathcal{E}(1)^{m+2} =\mathcal{K}^{-1} .
\]
The dual bundle of this root is denoted by $\mathcal{E}(-1)$. Then, for any integer $p\in\ZZ$, we have the~$p$th power $\mathcal{E}(p)$
of $\mathcal{E}(1)$. We call $p$ the {\em weight} of~$\mathcal{E}(p)$.
In particular, the canonical bundle $\mathcal{K}$ has weight $-(m+2)$, whereas
the anticanonical bundle $\mathcal{K}^{-1}$ has weight~$m+2$.

In general, the existence of an $(m+2)$nd root $\mathcal{E}(1)$ is restrictive to the global nature of the underlying CR structure on~$M$.
For the application of {\em tractor calculus} in CR geometry this assumption is basic.
In fact, the standard homogeneous model of CR geometry on the sphere allows a natural choice for $\mathcal{E}(1)$ (see~\cite{CG1}).
For our treatment of ${\rm spin}^\CC$ structures in CR geometry the choice of some~$\mathcal{E}(1)$ is useful as well.

\section{Pseudo-Hermitian geometry} \label{Sec3}

Let $\big(M^{2m+1}, H(M), J\big)$, $m\geq 1$, be strictly pseudoconvex.
Since $M$ is orientable, there exists some $1$-form $\theta$ on $M$, whose kernel $\operatorname{Ker}(\theta)$
defines the contact distribution~$H(M)$. The differen\-tial~${\rm d}\theta$ is a non-degenerate $2$-form on $H(M)$, and
the conditions \[ \iota_T\theta=\theta(T)=1\qquad\mbox{and}\qquad \iota_T{\rm d}\theta=0\]
define a unique vector field $T=T_\theta$, which is the {\em Reeb vector field} of $\theta$.
We use to call $T$ the {\em characteristic vector}. The tangent
bundle $T(M)$ splits into the direct sum
\[T(M) = H(M) \oplus \RR T\]
with corresponding projection $\pi_\theta\colon T(M)\to H(M)$. We use to say that vectors in $H(M)$ are {\em transverse}
(to the characteristic direction of~$T$).
Note that $\mathfrak{L}_T\theta=\mathfrak{L}_T{\rm d}\theta=0$ for the Lie derivatives
with respect to~$T_\theta$.

Furthermore,
\[
g_\theta(X,Y) := \frac{1}{2}{\rm d}\theta(X, JY),\qquad X,Y\in H(M),\]
defines a non-degenerate, symmetric bilinear form, i.e., a metric on $H(M)$, which is
either positive or negative definite.
In case $g_\theta$ is positive definite, we call $\theta\in\Omega^1(M)$ an adapted {\em pseudo-Hermitian structure} for $\big(M^{2m+1}, H(M), J\big)$. Note that any two pseudo-Hermitian structures $\theta$ and $\tilde{\theta}$
differ only by some positive function or {\em conformal scale}, i.e., $\tilde{\theta}={\rm e}^{2f}\theta$ for some $f\in C^\infty(M)$.

Let us fix some adapted pseudo-Hermitian structure $\theta$ on~$M$. To $\theta$ we have the {\em Webster--Tanaka connection}~$\nabla^{\rm W}$ on $T(M)$ (see~\cite{Tan,Web}),
for which by definition the characteristic vector~$T$, the metric $g_\theta$ and the complex structure~$J$ on $H(M)$ are parallel.
Hence, the {\em structure group} of~$\nabla^{\rm W}$ is the unitary group~${\rm U}(m)$.
In characteristic direction, we have
\begin{equation*} \label{WTC} \nabla^{\rm W}_{T}X = \frac{1}{2} ([T,X]-J[T,JX] )\end{equation*}
for $X\in\Gamma(H(M))$.

The torsion is given by some obligatory part
 \[ \nabla^{\rm W}_XY-\nabla^{\rm W}_YX - [X,Y] ={\rm d}\theta(X,Y)T \]
with transverse $X$, $Y$ in $H(M)$ and, furthermore, by
\begin{gather*}%\label{wttor}
\operatorname{Tor}^{\rm W}(T,X) = -\frac{1}{2} ([T,X]+J[T,JX] ),\qquad X\in H(M) . \end{gather*}
We call the latter part
\[
\tau(X) := \operatorname{Tor}^{\rm W}(T,X),\qquad X\in H(M),
\]
{\em Webster torsion tensor} of $\theta$ on $(M,H(M),J)$.
This is a symmetric and trace-free tensor. The composition $\tau\circ J=-J\circ \tau$ is symmetric and trace-free as well.
We set $\tau(X,Y)=g_\theta(\tau X,Y)$, $X,Y\in H(M)$.

As usual the curvature operator $R^{\rm W}(X,Y)$ of $\nabla^{\rm W}$ is defined by
\[
R^{\rm W}(X,Y) := \nabla^{\rm W}_X\nabla^{\rm W}_Y- \nabla^{\rm W}_Y\nabla^{\rm W}_X- \nabla^{\rm W}_{[X,Y]}
\]
for any $X,Y\in T(M)$. Since $\nabla^{\rm W}$ is metric,
$R^{\rm W}(X,Y)$ is skew-symmetric with respect to~$g_\theta$ on~$H(M)$.
The first Bianchi identity for $X,Y,Z\in H(M)$ is given by the cyclic sum
\begin{equation}\label{bianchi}
\sum_{XYZ}R^{\rm W}(X,Y)Z = \sum_{XYZ} {\rm d}\theta(X,Y)\tau(Z) .
\end{equation}

For any $X\in T(M)$, the {\em Webster--Ricci endomorphism} $\operatorname{Ric}^{\rm W}(X)$ is the $g_\theta$-trace of $R(X,\cdot)(\cdot)$, and
the {\em Webster scalar curvature} is the trace $\operatorname{scal}^{\rm W}=\operatorname{tr}_\theta \operatorname{Ric}^{\rm W}$ of the Webster--Ricci tensor on~$H(M)$.
On the other hand, the {\em pseudo-Hermitian Ricci form} is given by
\[
\rho_\theta(X,Y) := \frac{1}{2}\operatorname{tr}_\theta\big( g_\theta\big(R^{\rm W}(X,Y,J\cdot,\cdot)\big)\big)
\]
for any $X,Y\in H(M)$.
Then we have
\[
\operatorname{Ric}^{\rm W}(X,Y) = \rho_\theta(X,JY) + 2(m-1)\tau(X,JY)
\]
for any $X,Y\in H(M)$, where $\rho_\theta$ corresponds to the $J$-invariant and $\tau$ is the $J$-antiinvariant part of $\operatorname{Ric}^{\rm W}$ on~$H(M)$.

If $m\geq 2$ and the pseudo-Hermitian Ricci form $\rho_\theta$ is a multiple of ${\rm d}\theta$ we call $\theta$ a {\em pseudo-Einstein structure}
on the CR manifold $\big(M^{2m+1},H(M),J\big)$ (see~\cite{Lee}). For $m=1$ this condition is vacuous. However, for $m>1$ the pseudo-Einstein condition implies
 $\operatorname{Ric}^{\rm W}(T,JX)=\frac{1}{4m}X\big({\rm scal}^{\rm W}\big)$ for any~$X\in H(M)$. This is a suitable replacement for the Einstein condition when $m=1$ (see~\cite{CY}).
In any case we have $\rho_\theta=\frac{{\rm scal}^{\rm W}}{4m}{\rm d}\theta$ and the Webster scalar curvature of some pseudo-Einstein structure $\theta$
need not be constant. In fact, it
is constant if and only if
\[\operatorname{Ric}^{\rm W}(T) = \operatorname{tr}_\theta\big(\nabla^{\rm W}_{\cdot}\tau\big)(\cdot) = 0.\]

\section[Spin C structures]{$\boldsymbol{{\rm Spin}^\CC}$ structures} \label{Sec4}

Recall that the group ${\rm Spin}^\CC(2m)$ is a {\em central extension} of ${\rm SO}(2m)$ given by the exact sequence
\[
1 \to \ZZ_2 \to {\rm Spin}^\CC(2m) \to {\rm SO}(2m)\times {\rm U}(1) \to 1.
\]
This gives a twisted product ${\rm Spin}^\CC(2m)={\rm Spin}(2m)\times_{\ZZ_2}{\rm U}(1)$ with the spin group. We have
${\rm Spin}^\CC(2m)/{\rm U}(1)\cong {\rm SO}(2m)$, and a group homomorphism $\lambda\colon {\rm Spin}^\CC(2m)\to {\rm SO}(2m)$ as well as
${\rm Spin}^\CC(2m)/{\rm Spin}(2m)\cong {\rm U}(1)$. Note that there is also a canonical homomorphism
\[
j\colon \ {\rm U}(m) \to {\rm Spin}^{\CC}(2m),
\]
which is the lift of $\iota\times \det\colon {\rm U}(m)\to {\rm SO}(2m)\times {\rm U}(1)$.

Now let $\theta$ be a pseudo-Hermitian form on the strictly pseudoconvex CR manifold $\big(M^{2m+1},\allowbreak H(M),J\big)$, $m\geq 1$.
This gives rise to the metric $g_\theta$ on the Levi distribution $H(M)$.
We denote by ${\rm SO} (H(M))$ the principal ${\rm SO}(2m)$-bundle of orthonormal frames in $H(M)$.
A {\em ${\rm spin}^\CC$ structure} to~$\theta$ on~$M$ is a reduction
$(P,\Lambda)$ of the frame bundle~${\rm SO}(H(M))$. This means here, $P\to M$ is some principal ${\rm Spin}^\CC(2m)$-bundle with fiber bundle map
$\Lambda\colon P\to {\rm SO}(H(M))$
such that $\Lambda(p\cdot s)=\Lambda(p)\cdot \lambda(s)$ for all $p\in P$ and $s\in {\rm Spin}^\CC(2m)$.

Let $(P,\Lambda)$ be some fixed ${\rm spin}^\CC$ structure for $(M,\theta)$. Then $P_1:=P/{\rm Spin}(2m)\to M$ is a~principal ${\rm U}(1)$-bundle,
and we denote the associated complex line bundle by $L\to M$. This is the {\em determinant bundle} of the ${\rm spin}^\CC$ structure.
The corresponding fiber bundle map $\Lambda_1\colon P\to {\rm SO}(H(M))\times P_1$ over~$M$ is a twofold covering.
On the other hand, let~$L(\beta)\to M$ be a complex line bundle
determined by some integral class $\beta\in H^2(M,\ZZ)$. Then, if
\[
\beta \equiv -c_1(\mathcal{K}) \ \operatorname{mod} 2,
\]
there exists a ${\rm spin}^\CC$ structure $(P,\Lambda)$ to $\theta$ on $M$
with determinant bundle $L(\beta)$.

There exists always the {\em canonical ${\rm spin}^\CC$ structure} to $\theta$ on $M$, which stems from the lift $j\colon {\rm U}(m)\to {\rm Spin}^{\CC}(2m)$.
The corresponding determinant bundle is $\mathcal{K}^{-1}$. All other ${\rm spin}^\CC$ structures differ from the canonical one by multiplication
with a principal ${\rm U}(1)$-bundle, related to some line bundle $E(\alpha)$, $\alpha\in H^2(M,\ZZ)$.
The corresponding determinant bundle~$L(\beta)$ satisfies $E(\alpha)^2=\mathcal{K}\otimes L(\beta)$.
${\rm Spin}^\CC$ structures with the same determinant bundle~$L(\beta)$ are parametrized by the
elements in $H^1(M,\ZZ_2)$ (see~\cite{LMS,Pet}).

In particular, if $c_1(\mathcal{K})\equiv 0$ $\operatorname{mod} 2$, then $\theta$ on $M$ admits some ${\rm spin}^\CC$ structure with trivial determinant bundle.
This represents an {\em ordinary} spin structure for the Levi distribution~$H(M)$ with metric~$g_\theta$ (cf.~\cite{LeiN1}).
More generally, let us consider the powers $\mathcal{E}(p)$, $p\in\ZZ$, of an $(m+2)$nd root $\mathcal{E}(1)$ of $\mathcal{K}^{-1}$.
Then $-c_1(\mathcal{K})=(m+2)c_1(\mathcal{E}(1))$, and a ${\rm spin}^\CC$ structure for $(M,\theta)$ with determinant bundle $L=\mathcal{E}(p)$ exists when
\begin{equation*}\label{condi1} (m+2-p)c_1(\mathcal{E}(1)) \equiv 0\quad \operatorname{mod} 2.\end{equation*}

\begin{Lemma} Let $\mathcal{E}(1)$ be an $(m+2)$nd root of $\mathcal{K}^{-1}\to M^{2m+1}$. Then $\theta$ on $M$ admits a ${\rm spin}^\CC$ structure
with determinant bundle $L=\mathcal{E}(p)$, $p\in\ZZ$, if
\begin{enumerate}\itemsep=0pt
\item[$(i)$] $\mathcal{E}(1)$ itself admits some square root, or
\item[$(ii)$] $m$ and $p\in\ZZ$ are odd, or
\item[$(iii)$] $m$ and $p$ are even.
\end{enumerate}
\end{Lemma}

We say that a ${\rm spin}^\CC$ structure with determinant bundle $L=\mathcal{E}(p)$ has {\em weight} $p$.
In the following we assume that ${\rm spin}^\CC$ structures to $\theta$ on $M$ exist for all {\em necessary} weights $p\in\ZZ$.

\section{Spinors and connections} \label{Sec5}

Let $\big(M^{2m+1},H(M),J\big)$ be a strictly pseudoconvex CR manifold of hypersurface type and CR dimension $m\geq 1$,
and let $(P,\Lambda)$ be a ${\rm spin}^\CC$ structure of weight $\ell\in\ZZ$ for some given pseudo-Hermitian form $\theta$ on $M$.
The choice of $(P,\Lambda)$ gives rise to an
associated spinor bundle
\[ \Sigma(H(M)) := P\times_{\rho_{2m}}\Sigma\]
over $M$, where $\rho_{2m}$ denotes the representation of ${\rm Spin}^\CC(2m)$ on the complex spinor module $\Sigma$.
Note that the center ${\rm U}(1)$ acts by complex scalar multiplication on~$\Sigma$.
The spinor bundle has $\operatorname{rk}_\CC(\Sigma(H(M))=2^m$.

The spinor bundle $\Sigma(H(M))$ is equipped with a Hermitian inner product $\langle\cdot,\cdot\rangle$, and we have a Clifford multiplication
\begin{align*}
c \colon \ H(M) \otimes \Sigma(H(M)) &\to \Sigma(H(M)),\\
 (X,\phi) &\mapsto X\cdot \phi,
\end{align*}
which satisfies
\[
\langle X\cdot \psi,\phi\rangle = -\langle \psi,X\cdot \phi\rangle
\]
for any transverse $X\in H(M)$ and $\phi\in\Sigma(H(M))$, given at some point of~$M$.
The multiplica\-tion~$c$ extends to the complex Clifford bundle $\CC{\rm l}(H(M))$ of the Levi distribution.

The Webster--Tanaka connection $\nabla^{\rm W}$ to $\theta$ stems from a principal fiber bundle
connection on the unitary frame bundle, contained in ${\rm SO}(H(M))$. This gives rise to a covariant derivative~$\nabla^{\rm W}$
on any root of $\mathcal{K}^{-1}$ and its powers,
in particular, for $\mathcal{E}(1)$ and the determinant bundle $L=\mathcal{E}(\ell)$.

Recall that $(P,\Lambda)$ induces a twofold covering map $P\to {\rm SO}(H(M))\times P_1$.
Then the Webster--Tanaka connection lifts to~$P$, which in turn gives rise to some covariant derivative on spinor
fields:
\begin{align*}
\nabla^\Sigma \colon \ \Gamma(T(M)) \otimes \Gamma(\Sigma(H(M))&\to \Gamma(\Sigma(H(M)),\\
 (X,\phi) &\mapsto \nabla^\Sigma_X\phi.
\end{align*}
Note that this construction does not need an {\em auxiliary connection} on the determinant bundle~$L$.
We only use the Webster--Tanaka connection on~$L$ and call~$\nabla^\Sigma$
the {\em Webster--Tanaka spinor derivative} to the given ${\rm spin}^\CC$ structure of weight~$\ell$.

The spinor derivative satisfies the rules
\begin{gather*}
\nabla^\Sigma_Y (X\cdot \phi) = \big(\nabla^{\rm W}_YX\big)\cdot \phi + X\cdot \nabla^\Sigma_Y\phi\qquad \mbox{and} \qquad
Y\langle \phi,\psi\rangle = \big\langle \nabla^\Sigma_Y \phi,\psi\big\rangle + \big\langle \phi,\nabla_Y^\Sigma \psi\big\rangle
\end{gather*}
for any $X\in\Gamma(H(M))$, $Y\in\Gamma(T(M))$ and $\phi,\psi\in\Gamma(\Sigma(H(M)))$.
Locally, with respect to some orthonormal frame $s=(s_1,\dots,s_{2m})$, the spinor derivative is given by the formula
\[
\nabla^\Sigma\phi = {\rm d}\phi + \frac{1}{2}\sum_{j<k}^{2m}g_\theta\big(\nabla^{\rm W}s_j,s_k\big)s_js_k\cdot \phi + \frac{1}{2}A^{{\rm W},s} \phi,
\]
where $A^{{\rm W},s}$ denotes the local Webster--Tanaka connection form on~$P_1$ with values in~${\rm i}\RR$.
The curvature~$R^\Sigma$ of the spinor derivative~$\nabla^\Sigma$ is then given by
\begin{align*}
R^\Sigma(X,Y)\phi &= \nabla^\Sigma_X\nabla^\Sigma_Y\phi- \nabla^\Sigma_Y\nabla^\Sigma_X\phi-\nabla^\Sigma_{[X,Y]}\phi\\
&= \frac{1}{4}\sum_{j,k=1}^{2m} g_\theta\big(R^{\rm W}(X,Y)s_j,s_k\big)s_js_k\cdot \phi + \frac{1}{2}{\rm d}A^{\rm W}(X,Y)\phi
\end{align*}
for any $X,Y\in TM$ and spinor $\phi\in\Gamma(\Sigma(H(M)))$. Note that
\begin{equation}\label{curline}
{\rm d}A^{\rm W} = \frac{-{\rm i}\ell}{m+2}\rho_\theta
\end{equation}
is a multiple of the pseudo-Hermitian Ricci form $\rho_\theta$.

The underlying pseudo-Hermitian form $\theta$ gives rise to further structure on the spinor bundle $\Sigma(H(M))$.
In fact, recall that ${\rm d}\theta$ is $\nabla^{\rm W}$-parallel and {\em basic}, i.e., $\iota_T{\rm d}\theta=0$.
We set $\Theta:=\frac{{\rm i}\,{\rm d}\theta}{2}\in \CC{\rm l}(H(M))$ in the complex Clifford bundle. Then
$\Theta$ acts by real eigenvalues $\mu_q=m-2q$, $q\in\{0,\dots,m\}$ on $\Sigma(H(M))$. We obtain
the decomposition
\[
\Sigma(H(M)) = \bigoplus_{q=0}^m \Sigma^{\mu_q}(H(M))
\]
into $\Theta$-eigenspaces $\Sigma^{\mu_q}(H(M))$ of rank
$\left( \begin{smallmatrix} m\\ q\end{smallmatrix}\right)$ to the eigenvalue~$\mu_q$ (see~\cite{Pet}).
We call the bundles to the $\Theta$-eigenvalues $\mu_q=\pm m$ {\em extremal}.
(We also define $\Sigma^{a}=\{0\}$ to be trivial for any $a>m$ and $a<-m$.)
Accordingly, we can decompose any spinor~$\phi$ on~$M$
into
\[
\phi = \sum_{q=0}^m \phi_{\mu_q},
\]
where $\Theta\phi_{\mu_q}=(m-2q)\phi_{\mu_q}$. This decomposition is compatible with the spinor derivative $\nabla^\Sigma$.

\section{Kohn--Dirac and twistor operators} \label{Sec6}

Let $\big(M^{2m+1},H(M),J\big)$, $m\geq 1$, be strictly pseudoconvex. We have
$H(M)\otimes\CC=T_{10}\oplus T_{01}$ and any real transverse vector $X\in H(M)$ can be written as $X=X_{10}+X_{01}$ with
\[
X_{10}= \frac{X-{\rm i}JX}{2}\in T_{10}\qquad \mbox{and}\qquad X_{01}=\frac{X+{\rm i}JX}{2}\in T_{01}.
\]
If $e=(e_1,\dots,e_m)$ denotes a complex orthonormal basis of $(H,J,g_\theta)$, i.e., $s=(e_1,Je_1,\dots,e_m,\allowbreak Je_m)$ is a real orthonormal basis of $(H,g_\theta)$,
we set \[E_\alpha := (e_\alpha)_{10} = \frac{e_\alpha-{\rm i}Je_\alpha}{2},\qquad\alpha=1,\dots,m.\]
The vectors $(E_1,\dots,E_m)$ form an orthogonal basis with respect to the {\em Levi form} on $T_{10}$.
As elements in the complexified Clifford algebra $\CC{\rm l}(H(M))$ we have $E_\alpha E_\alpha=0$ and $E_\alpha\overline{E_\beta}+\overline{E_\beta}E_\alpha=-\delta_{\alpha\beta}$
for any $\alpha,\beta=1,\dots,m$. Moreover,
\begin{equation*}%\label{sekkel}
\sum_{\alpha=1}^m \overline{E_\alpha}E_\alpha = -\frac{1}{2}(m+\Theta),\qquad
\sum_{\alpha=1}^m E_\alpha\overline{E_\alpha} = -\frac{1}{2}(m-\Theta).
\end{equation*}

Now let $\theta$ be a pseudo-Hermitian form on $M$ with ${\rm spin}^\CC$ structure of weight $\ell\in\ZZ$.
The spinorial derivative $\nabla^\Sigma$ on $\Sigma(H(M))$ is induced by the Webster--Tanaka connection.
In the following, we allow covariant derivatives with respect to $Z\in H(M)\otimes\CC$.
This is defined by $\CC$-linear extension and denoted by $\nabla^{\rm tr}_Z$.
This derivative in transverse direction decomposes into
\[
\nabla^{\rm tr} = \nabla_{10} \oplus \nabla_{01},
\]
i.e., for any spinor $\phi\in\Gamma(\Sigma(H(M)))$, we have locally
\[
\nabla_{10}\phi = \sum_{\alpha=1}^m E_\alpha^*\otimes \nabla^{\rm tr}_{E_\alpha}\phi\qquad\mbox{and}\qquad
\nabla_{01}\phi = \sum_{\alpha=1}^m \overline{E_\alpha}^*\otimes \nabla^{\rm tr}_{\overline{E_\alpha}}\phi
\]
with respect to some frame $(E_1,\dots,E_m)$ of $T_{10}$.

Recall that Clifford multiplication is denoted by $c$. Then we can define the first order differential operators
\[
D_-\phi = c(\nabla_{10}\phi)\qquad \mbox{and}\qquad D_+\phi = c(\nabla_{01}\phi)
\]
for spinors $\phi\in\Gamma(\Sigma(H(M)))$. Locally, the two operators are given by
\[
D_-\phi = 2\sum_{\alpha=1}^m \overline{E_\alpha}\cdot\nabla^{\rm tr}_{E_\alpha}\phi\qquad\mbox{and}\qquad
D_+\phi = 2\sum_{\alpha=1}^m E_\alpha\cdot \nabla^{\rm tr}_{\overline{E_\alpha}}\phi.
\]
Note that $\Theta\cdot X_{10}- X_{10}\cdot \Theta = -2X_{10}$ for $X_{10}\in T_{10}$. This shows
\[ T_{10}\cdot \Sigma^{\mu_q} \subseteq \Sigma^{\mu_{q+1}}\qquad\mbox{and}\qquad T_{01}\cdot \Sigma^{\mu_q} \subseteq \Sigma^{\mu_{q-1}}\]
for any $q\in\{0,\dots,m\}$. Hence, the operator $D_+$ maps spinors from $\Gamma\big(\Sigma^{\mu_{q}}\big)$ to $\Gamma\big(\Sigma^{\mu_{q+1}}\big)$.
Similarly, $D_-\colon \Gamma\big(\Sigma^{\mu_{q}}\big)\to\Gamma\big(\Sigma^{\mu_{q-1}}\big)$.
In fact, we have $[\Theta,D_+]=-2D_+$ and $[\Theta,D_-]= 2D_-$.

We compute the square of $D_+$. Locally, around any $p\in M$, we can choose a {\em synchronized frame} of the form $(e_1,\dots,e_m)$
with
\[
\nabla^{\rm W}_{e_\alpha}e_\beta(p) = 0,\qquad \alpha,\beta\in\{1,\dots,m\}.
\]
Then
\begin{align*}
(D_+)^2\phi &=
4\sum_{\alpha,\beta=1}^m E_\alpha E_\beta\nabla^{\rm tr}_{\overline{E_\alpha}}\nabla^{\rm tr}_{\overline{E_\beta}}\phi =
2\sum_{\alpha,\beta} E_\alpha E_\beta\cdot R^{\Sigma}\big(\overline{E_\alpha},\overline{E_\beta}\big)\phi\\
&= -2\bigg( \sum_{\alpha,\beta} E_\alpha\tau(\overline{E_\alpha})E_\beta\overline{E_\beta}+ E_\alpha\overline{E_\alpha}E_\beta\tau\big(\overline{E_\beta}\big) \bigg)\phi = 0,
\end{align*}
where we use (\ref{bianchi}), (\ref{curline}) and the fact that $\tau$, $\tau\circ J$ are trace-free.
Similarly, we obtain $D_-^2=0$. Thus, we have constructed two chain complexes
\begin{equation}\label{spincpx}
0\to \Gamma\big(\Sigma^{\mu_0}\big)\stackrel{D_+}{\longrightarrow}\Gamma\big(\Sigma^{\mu_1}\big)\stackrel{D_+}{\longrightarrow} \cdots\stackrel{D_+}{\longrightarrow}\Gamma\big(\Sigma^{\mu_{m-1}}\big)
\stackrel{D_+}{\longrightarrow}\Gamma\big(\Sigma^{\mu_m}\big)\to 0
\end{equation}
and
\[
0\to \Gamma\big(\Sigma^{\mu_m}\big)\stackrel{D_-}{\longrightarrow} \Gamma\big(\Sigma^{\mu_{m-1}}\big)\stackrel{D_-}{\longrightarrow}\cdots\stackrel{D_-}{\longrightarrow}\Gamma\big(\Sigma^{\mu_{1}}\big)
\stackrel{D_-}{\longrightarrow}\Gamma\big(\Sigma^{\mu_0}\big)\to 0.
\]
From the discussions in Section~\ref{Sec10} it will become clear that these complexes produce finite dimensional cohomology groups.
This compares to the construction of {\em spinorial cohomology} on K\"ahler manifolds as described in~\cite{Mic}.

Next we define
\[
D_\theta\phi = c\big(\nabla_{\cdot}^{\rm tr}\phi\big) = (D_++D_-)\phi.
\]
This \looseness=-1 is a first order, subelliptic differential operator acting on spinor fields $\phi\in\Gamma(\Sigma(H(M))$.
We call $D_\theta$ the {\em Kohn--Dirac operator} to $\theta$ with ${\rm spin}^\CC$ structure of weight~$\ell$ on~$M$ (see~\cite{Pet}; cf.~\cite{LeiN1, Stadt}).
Locally, with respect to an orthonormal frame $(s_1,\dots,s_m)$, the Kohn--Dirac operator is given by
\[
D_\theta\phi = \sum_{i=1}^{2m} s_i\cdot \nabla^{\rm tr}_{s_i}\phi.
\]

Obviously, $D_\theta$ does not preserve the decomposition of spinors with respect to $\Theta$-eigenvalues.
However, we have the identity
\[
D^2_\theta = D_+D_- + D_-D_+,
\]
which shows that the square of the Kohn--Dirac operator maps sections
of $\Sigma^{\mu_q}(H(M))$ to sections of $\Sigma^{\mu_q}(H(M))$ again, i.e.,
\[D_\theta^2\colon \ \Gamma\big(\Sigma^{\mu_q}\big)\to \Gamma\big(\Sigma^{\mu_q}\big),\qquad q=0,\dots,m.\]

On the spinor bundle, we have the $L_2$-inner product defined by
\[
(\phi,\psi) := \int_{M} \langle\phi,\psi\rangle \operatorname{vol}_\theta
\]
for compactly supported spinors $\phi,\psi\in\Gamma_c(\Sigma)$, where
\[
\operatorname{vol}_\theta := \theta\wedge ({\rm d}\theta)^m
\]
denotes the induced volume form of the pseudo-Hermitian structure $\theta$ on~$M$.
The Kohn--Dirac operator $D_\theta$ is formally self-adjoint with respect to this $L_2$-inner product $(\cdot,\cdot)$ on $\Gamma_c(\Sigma)$ (see~\cite{LeiN1}).

Complementary to the Kohn--Dirac operator $D_\theta$, we have {\em twistor operators} $P^{(\mu_q)}$ acting on $\Gamma\big(\Sigma^{\mu_q}(H(M))\big)$ for $q=0,\dots,m$.
In fact, there are orthogonal decompositions
\[
T_{10}^*\otimes \Sigma^{\mu_q} \cong \operatorname{Ker}(c)\oplus \Sigma^{\mu_{q-1}}\qquad\mbox{and}\qquad T_{01}^*\otimes \Sigma^{\mu_q} \cong \operatorname{Ker}(c)\oplus \Sigma^{\mu_{q+1}},
\]
where $\operatorname{Ker}(c)$ denote the corresponding kernels of the Clifford multiplication. Then
with
\[
a_q:= \frac{1}{2(q+1)}\qquad\mbox{and}\qquad b_q:=\frac{1}{2(m-q+1)}
\]
we have for the derivatives $\nabla_{10}\phi_{\mu_q}$ and $\nabla_{01}\phi_{\mu_q}$ of a spinor $\phi_{\mu_q}\in\Gamma\big(\Sigma^{\mu_q}\big)$ the decompositions
\begin{gather*}
\nabla_{10}\phi_{\mu_q} = P_{10}\phi_{\mu_q} - b_q\sum_{\alpha=1}^m E_\alpha^*\otimes E_\alpha\cdot D_-\phi_{\mu_q},\\
\nabla_{01}\phi_{\mu_q} = P_{01}\phi_{\mu_q} - a_q\sum_{\alpha=1}^m \overline{E_\alpha}^*\otimes \overline{E_\alpha}\cdot D_+\phi_{\mu_q},
\end{gather*}
where the twistor operators map to $\operatorname{Ker}(c)$ by
\begin{gather*}
P_{10}(\phi_{\mu_q}) = \sum_{\alpha=1}^m E_\alpha^*\otimes \left( \nabla_{E_\alpha}\phi_{\mu_q} + b_q E_\alpha\cdot D_-\phi_{\mu_q}\right),\\
P_{01}(\phi_{\mu_q}) = \sum_{\alpha=1}^m \overline{E_\alpha}^*\otimes \left( \nabla_{\overline{E_\alpha}}\phi_{\mu_q} +
a_q \overline{E_\alpha}\cdot D_+\phi_{\mu_q}\right),
\end{gather*}
respectively. The sum $P^{(\mu_q)}=P_{10}+P_{01}$ is given with respect to a local orthonormal
frame $s$ by
\[
P^{(\mu_q)}\phi_{\mu_q} = \sum_{i=1}^{2m} s_i^*\otimes\left( \nabla^{\rm tr}_{s_i}\phi_{\mu_q}+a_q\frac{s_i+{\rm i}Js_i}{2}D_+\phi_{\mu_q} +b_q\frac{s_i-{\rm i}Js_i}{2}D_-\phi_{\mu_q} \right).
\]
This is the projection of $\nabla^{\rm tr}\phi_{\mu_q}$ to the kernel $\operatorname{Ker}(c)$.

\section{Covariant components and spinorial CR invariants} \label{Sec7}

In the previous section we have introduced Kohn--Dirac operators $D_\theta$ and twistor operators $P^{(\mu_q)}$ for
${\rm spin}^\CC$ structures of weight $\ell\in\ZZ$. We have only used the Webster--Tanaka connection for their construction.
Now we compute the transformation law for~$D_\theta$ and $P^{(\mu_q)}$ under conformal change of the pseudo-Hermitian structure.
It turns out that certain components of $D_\theta$ and $P^{(\mu_q)}$ are CR invariants.

Let $\theta$ and $\tilde{\theta}= {\rm e}^{2f}\theta$ be two adapted pseudo-Hermitian structures on $\big(M^{2m+1},H(M),J\big)$, $m\geq 1$.
We denote by $\nabla^{\rm W}$ and~$\nabla^\Sigma$ derivatives with respect to~$\theta$.
The derivatives with respect to~$\tilde{\theta}$ are simply denote by $\widetilde{\nabla}$.
Note that the structure group of the Webster--Tanaka connection is~${\rm U}(m)$ for any pseudo-Hermitian form.
We have the transformation rule
\begin{gather}
\widetilde{\nabla}_{X_{10}}Y = \nabla^{\rm W}_{X_{10}}Y + 2X_{10}(f)Y_{10} + 2Y_{10}(f)X_{10} - 2g_\theta(X_{10},Y_{01})\operatorname{grad}_{01}(f),\nonumber\\
\widetilde{\nabla}_{X_{01}}Y = \nabla^{\rm W}_{X_{01}}Y + 2X_{01}(f)Y_{01} + 2Y_{01}(f)X_{01}- 2g_\theta(X_{01},Y_{10})\operatorname{grad}_{10}(f),\label{traflaw1}
\end{gather}
where $X=X_{10}+X_{01}$ and $Y=Y_{10}+Y_{01}$ are transverse vectors (see, e.g.,~\cite{Lee}). The gradient $\operatorname{grad}_\theta(f)\in \Gamma(H(M))$ with complex components
 $\operatorname{grad}_{10}(f)\in\Gamma(T_{10})$ and $\operatorname{grad}_{01}(f)\in\Gamma(T_{01})$
is dual via $g_\theta$ to the restriction of the differential~${\rm d}f$ to~$H(M)$.

Now let $(P,\Lambda)$ be some ${\rm spin}^\CC$ structure of weight $\ell\in\ZZ$ to $\theta$ on $\big(M^{2m+1},H(M),J\big)$.
The canonical bundle $\mathcal{K}$ and all line bundles $\mathcal{E}(p)$, $p\in\ZZ$, are {\em natural} for the underlying
CR structure. In particular, the determinant bundle $L\to M$ of weight~$\ell$ is natural, and the corresponding principal ${\rm U}(1)$-bundles~$P_1$ and~$\tilde{P}_1$
of frames in~$L$ with respect to $\theta$ and $\tilde{\theta}$, respectively, are naturally identified.
The same is true for the orthonormal frames in~$H(M)$ to $\theta$ and $\tilde{\theta}$.
Thus, there exists a unique ${\rm spin}^\CC$ structure $\big(\tilde{P},\tilde{\Lambda}\big)$ with respect to~$\tilde{\theta}$
on~$M$, whose {\em spinor frames} are naturally identified with those of~$(P,\Lambda)$.
Of course, the determinant bundle to $\big(\tilde{P},\tilde{\Lambda}\big)$ has weight~$\ell$ again, and there exists
an unitary isomorphism
\begin{align*}
\Sigma(H(M))&\cong \widetilde{\Sigma}(H(M)),\\
\phi&\mapsto \tilde{\phi},
\end{align*}
between the two kinds of spinor bundles such that
$X\cdot\phi$ is sent to ${\rm e}^{-f}X \tilde{\cdot} \tilde{\phi}$ for any transverse vector $X\in H(M)$ and spinor $\phi\in \Sigma(H(M))$.
Also note that $\widetilde{\Theta}\tilde{\phi}=(\Theta\phi)^{\widetilde{}}$, i.e., the decomposition $\Sigma(H(M))=\oplus_{q=0}^m \Sigma^{\mu_q}(H(M))$
into $\Theta$-eigenspaces is CR-invariant.

We compare now the spinor derivatives with respect to $\theta$ and $\tilde{\theta}$, respectively.
First, let
$\sigma= E_1\wedge\dots\wedge E_m$ be a local section in~$\mathcal{K}^{-1}\to M$
and $\tilde{\sigma}= \widetilde{E_1}\wedge\dots\wedge \widetilde{E_m}$ the corresponding section with respect to~$\tilde{\theta}$.
Then $\tilde{\sigma}={\rm e}^{-mf}\sigma$ and with~(\ref{traflaw1}) we obtain the transformation rule
\begin{gather*}
\widetilde{\nabla}_{X_{10}}\tilde{\sigma} = \left(A^\sigma(X_{10}) + (m+2)X_{10}(f) \right) \tilde{\sigma} = A^{\tilde{\sigma}}(X_{10})\tilde{\sigma},\\
\widetilde{\nabla}_{X_{01}}\tilde{\sigma} = \left( A^\sigma(X_{01}) - (m+2)X_{01}(f) \right) \tilde{\sigma} = A^{\tilde{\sigma}}(X_{01})\tilde{\sigma},
\end{gather*}
$X=X_{10}+X_{01}$, for the local connections forms of $\mathcal{K}^{-1}$. This gives
\[ A^{\tilde{\sigma}}(X)- A^{\sigma}(X)=-{\rm i}(m+2) (JX)(f), \qquad X\in H(M).\]
Accordingly,
\[
A^{\tilde{\sigma}}(X)- A^{\sigma}(X) = -{\rm i}\ell (JX)(f),\qquad X\in H(M),
\]
 for the local connection forms on $L=\mathcal{E}(\ell)$.
With formulas in~\cite{LeiN1} we obtain for spinors~$\phi$ the transformation rule
\begin{gather*}
\widetilde{\nabla}_{X_{10}}\tilde{\phi} = \widetilde{\nabla_{X_{10}}^{\Sigma}\phi}-(X_{10}\cdot \operatorname{grad}_{01}(f)\cdot \phi)^{\widetilde{}}
+\frac{\ell-2}{2}X_{10}(f)\tilde{\phi}
-\frac{1}{2}X_{10}(f)(\Theta\phi)^{\widetilde{}},
\\
\widetilde{\nabla}_{X_{01}}\tilde{\phi} = \widetilde{\nabla_{X_{01}}^{\Sigma}\phi}-(X_{01}\cdot \operatorname{grad}_{10}(f)\cdot \phi)^{\widetilde{}}
-\frac{\ell+2}{2}X_{01}(f)\tilde{\phi}
+\frac{1}{2}X_{01}(f)(\Theta\phi)^{\widetilde{}}.
\end{gather*}
This gives for $\phi_{\mu_q}\in\Gamma(\Sigma^{\mu_q})$, $q\in\{0,\dots,m\}$,
\begin{gather*}
\tilde{D}_-\tilde{\phi}_{\mu_q} = {\rm e}^{-f}\left( D_-\phi_{\mu_q}
+ \left(m+1+\frac{\mu_q+\ell}{2}\right) \operatorname{grad}_{01}(f) \phi_{\mu_q} \right)^{\widetilde{}},
\\
\tilde{D}_+\tilde{\phi}_{\mu_q} = {\rm e}^{-f}\left( D_+\phi_{\mu_q}
+ \left(m+1-\frac{\mu_q+\ell}{2}\right) \operatorname{grad}_{10}(f) \phi_{\mu_q} \right)^{\widetilde{}},
\end{gather*}
and we obtain
\begin{gather*}
\tilde{D}_- \big({\rm e}^{-v_- f}\tilde{\phi}_{\mu_q}\big) = {\rm e}^{-(v_-+1)f} \widetilde{D_-\phi_{\mu_q}}\qquad
\mbox{for}\quad v_-=m+1+\frac{\mu_{q}+\ell}{2},\\
\tilde{D}_+ \big({\rm e}^{-v_+ f}\tilde{\phi}_{\mu_q}\big) = {\rm e}^{-(v_++1)f} \widetilde{D_+\phi_{\mu_q}}\qquad
\mbox{for}\quad v_+=m+1-\frac{\mu_{q}+\ell}{2}.
\end{gather*}
Hence, for the $\Theta$-eigenvalue $\mu_q=-\ell$, we have
\begin{equation}\label{compare}
D_{\tilde{\theta}}\big({\rm e}^{-(m+1)f}\tilde{\phi}_{-\ell}\big) = {\rm e}^{-(m+2)f} \widetilde{D_\theta\phi_{-\ell}},
\end{equation}
i.e., the restriction of the Kohn--Dirac operator $D_\theta$ of weight $\ell$ to $\Gamma\big(\Sigma^{-\ell}\big)$ acts CR-covariantly.

Recall that the given ${\rm spin}^\CC$ structure of weight $\ell\in\ZZ$ on $M$ is determined by some complex line bundle $E(\alpha)$, $\alpha\in H^2(M,\ZZ)$.

\begin{Definition} Let $\big(M^{2m+1},H(M),J\big)$, $m\geq 1$, be strictly pseudoconvex
with pseudo-Her\-mi\-tian form~$\theta$ and ${\rm spin}^\CC$ structure of weight $\ell\in\ZZ$.\samepage
\begin{enumerate}\itemsep=0pt
\item[(a)] A spinor $\phi\in\Gamma(\Sigma(H(M))$ in the kernel of the Kohn--Dirac operator, i.e., $D_\theta\phi=0$,
is called \emph{harmonic}. We denote by $\mathcal{H}^{q}(\alpha)$ the space of harmonic spinors
with $\Theta$-eigenvalue~$\mu_q$, $q\in\{0,\dots,m\}$. Its dimension is denoted $h_{q}(\alpha)$.
\item[(b)] For weight $\ell\in\{-m,-m+2,\dots,m-2,m\}$ the differential operator
\[
\mathcal{D}_\ell\colon \ \Gamma\big(\Sigma^{-\ell}\big) \to \Gamma\big(\Sigma^{-\ell+2}\big)\oplus\Gamma\big(\Sigma^{-\ell-2}\big)
\]
denotes the restriction of $D_\theta$ (of weight $\ell$) to spinors of $\Theta$-eigenvalue $-\ell$.
We call $\mathcal{D}_\ell$ the $\ell$th (CR-covariant) component of the Kohn--Dirac operator.
\item[(c)] A spinor $\phi$ in the kernel of $\mathcal{D}_\ell$ is called \emph{harmonic of weight~$\ell$}.
\end{enumerate}
\end{Definition}

Since $\mathcal{D}_\ell$ acts by (\ref{compare}) CR-covariantly,
harmonic spinors of weight $\ell$ are CR invariants of $(H(M),J)$ on~$M$.
The dimension $h_{\frac{m+\ell}{2}}(\alpha)$ is a CR invariant as well.
In Section~\ref{Sec10} we will see that in fact all dimensions $h_{q}(\alpha)$, $0<q<m$, are CR-invariant numbers.

Let us consider the twistor operators $P_{10}$ and $P_{01}$. Calculating as above we find
\[\widetilde{\nabla}_{X_{10}}\big({\rm e}^{-w_-f}\tilde{\phi}_{\mu_q}\big)+b_qX_{10}\tilde{D}_-\big({\rm e}^{-w_-f}\tilde{\phi}_{\mu_q}\big) =
{\rm e}^{-w_-f}\cdot\big(\nabla_{X_{10}}\phi_{\mu_q}+b_qX_{10}D_-\phi_{\mu_q}\big)^{\widetilde{}}\]
exactly for $w_-=\frac{\ell-\mu_q}{2}-1$, and
\[\widetilde{\nabla}_{X_{01}}\big({\rm e}^{-w_+f}\tilde{\phi}_{\mu_q}\big)+a_qX_{01}\tilde{D}_+\big({\rm e}^{-w_+f}\tilde{\phi}_{\mu_q}\big) =
{\rm e}^{-w_+f}\cdot\big(\nabla_{X_{01}}\phi_{\mu_q}+a_qX_{01}D_-\phi_{\mu_q}\big)^{\widetilde{}}\]
exactly for $w_+=\frac{\mu_q-\ell}{2}-1$.
Hence, the twistor operator $P^{(\mu_q)}=P_{10}+P_{01}$ is CR-covariant for the $\Theta$-eigenvalue $\mu_q=\ell$.

\begin{Definition} Let $\big(M^{2m+1},H(M),J\big)$, $m\geq 1$, be strictly pseudoconvex
with pseudo-Her\-mi\-tian form $\theta$ and ${\rm spin}^\CC$ structure of weight $\ell\in\ZZ$.
\begin{enumerate}\itemsep=0pt
\item[(a)] For weight $\ell\in\{-m,-m+2,\dots,m-2,m\}$
the differential operator
\[
\mathcal{P}_\ell\colon \ \Gamma\big(\Sigma^{\ell}\big) \to \Gamma(\operatorname{Ker}(c)) \subset \Gamma\big(H(M)\otimes \Sigma^{\ell}\big)
\]
denotes the $\ell$th component of the twistor operator to $\theta$.

\item[(b)] A (non-trivial) element in the kernel of $\mathcal{P}_\ell$ is called \emph{CR twistor spinor of weight~$\ell$}.
The dimension $p_\ell$ of $\operatorname{Ker}(\mathcal{P}_\ell)$ denotes a CR invariant.
\end{enumerate}
\end{Definition}

In the non-extremal cases, i.e., for $\ell\neq \pm m$, the twistor equation is overdetermined.
In fact, similar as in~\cite{Pil}, we suppose the existence of a~{\em twistor connection} such that CR twistor spinors correspond to parallel sections in certain {\em twistor bundles}. This would
imply that for $\ell\neq \pm m$ the CR invariants $p_\ell$ are numbers.

\begin{Example} Parallel spinors of weight $\ell\in\{-m,-m+2,\dots,m-2,m\}$ are CR twistor spinors. For the spin case ($\ell=0$)
we discuss parallel spinors to any eigenvalue $\mu_q$ in~\cite{LeiN2}. They occur on pseudo-Einstein spin manifolds. In Section \ref{Sec12} we
demonstrate the construction of closed CR manifolds admitting parallel spinors with $\ell=\mu_q=0$, i.e., CR twistor spinors.
\end{Example}

\begin{Example}
In \cite{LeiN1} we describe {\em pseudo-Hermitian Killing spinors} for the case of a spin structure on $M$.
These spinors are in the kernel of the twistor operator and realize a certain lower bound for the non-zero eigenvalues of the Kohn--Dirac operator~$D_\theta$.
In particular, we find CR twistor spinors of weight $\ell=0$ on the standard spheres $S^{2m+1}$ of even CR dimension $m\geq 2$.

We also find Killing spinors in case that $\theta$
is related to some {\em $3$-Sasakian structure} on $M$. However, in this situation the CR dimension $m$ is odd and
$\ell=0$ is impossible. Such Killing spinors are not CR twistor spinors.
\end{Example}

\section{Vanishing theorems for harmonic spinors} \label{Sec8}

Let $\big(M^{2m+1}, H(M), J\big)$, $m\geq 1$, be strictly pseudoconvex with pseudo-Hermitian structure $\theta$ and Kohn--Dirac operator~$D_\theta$ to some ${\rm spin}^\CC$ structure of weight $\ell\in\ZZ$.
The operator~$D_\theta$ is formally
self-adjoint and there exists a Schr\"odinger--Lichnerowicz-type formula (see~\cite{Pet}; cf.~\cite{LeiN1}).
We use this formula to derive {\em vanishing theorems} for harmonic spinors.

Let $\nabla^{\rm tr}$ denote the transversal part of the Webster--Tanaka spinorial derivative to a chosen
${\rm spin}^\CC$ structure of weight $\ell\in\ZZ$.
Then $\Delta^{\rm tr}=-\operatorname{tr}_{\theta} \big(\nabla^{\rm tr}\circ\nabla^{\rm tr}\big)$ denotes the {\em spinor sub-Laplacian}, and
we have
\[
\Delta^{\rm tr} = \nabla^*_{10}\nabla_{10} + \nabla^*_{01}\nabla_{01},
\]
where $\nabla^*_{10}$ and $\nabla^*_{01}$ are the formal adjoint to $\nabla_{10}$ and $\nabla_{01}$, respectively.
As in~\cite{LeiN1} we obtain with~(\ref{curline}) the equation
\begin{equation*}%\label{between}
D_\theta^2\phi = \Delta^{\rm tr}\phi - \frac{{\rm i}\ell}{2(m+2)}\rho_\theta\phi + \frac{1}{4}{\rm scal}^{\rm W}\cdot\phi - {\rm d}\theta\nabla^{\Sigma}_T\phi
\end{equation*}
for the square of the Kohn--Dirac operator. For the spinorial derivative in characteristic direction we compute
\begin{gather*}%\label{transdivric}
\nabla^\Sigma_T\phi = \frac{{\rm i}}{4m}\left( 2\nabla^*_{10}\nabla_{10}-2\nabla^*_{01}\nabla_{01} + {\rm i}\rho_\theta - \frac{\ell \operatorname{scal}^{\rm W}}{2(m+2)}
\right)(\phi).
\end{gather*}
This results in the Schr\"odinger--Lichnerowicz-type formula (cf.~\cite{LeiN1, Pet})
\begin{gather}\nonumber
D_\theta^2 = \left( \left(1-\frac{\Theta}{m}\right) \nabla^*_{10}\nabla_{10} + \left(1+\frac{\Theta}{m}\right) \nabla^*_{01}\nabla_{01}\right)\\
\hphantom{D_\theta^2 =}{} - \frac{{\rm i}}{2}\left( \frac{\ell}{m+2}+\frac{\Theta}{m}\right)\rho_\theta + \left(1+\frac{\ell\Theta}{m(m+2)}\right)\frac{{\rm scal}^{\rm W}}{4}.\label{formtheo}
\end{gather}

The curvature part in~(\ref{formtheo}) acts by Clifford multiplication as a self-adjoint operator on spinors. In case that this operator is positive definite, at each point
of some closed manifold~$M$,
we immediately obtain vanishing results for harmonic spinors in the non-extremal bundles $\Sigma^{\mu_q}$ ($\mu_q\neq \pm m$).
We aim to specify the situation.
Let us call the pseudo-Hermitian Ricci form $\rho_\theta$ {\em positive (resp.\ negative) semidefinite} if all eigenvalues are nonnegative (resp. nonpositive) on~$M$.
We can state our basic vanishing result as follows.

\begin{Proposition} \label{vani} Let $\theta$ be some pseudo-Hermitian structure on a closed CR manifold $M^{2m+1}$,
$m\geq 1$, with ${\rm spin}^\CC$ structure of weight $\ell\in \ZZ$.
\begin{enumerate}\itemsep=0pt
\item[$(a)$] A non-extremal bundle $\Sigma^{\mu_q}$ allows no harmonic spinors under the following conditions:
\begin{enumerate}\itemsep=0pt
\item[$(1)$] $\mu_q=-\frac{m\ell}{m+2}$ and ${\rm scal}^{\rm W}\geq 0$ on $M$ with ${\rm scal}^{\rm W}(p)>0$ at some point $p\in M$.
\item[$(2)$] $\mu_q>-\frac{m\ell}{m+2}$, $\rho_\theta$ semidefinite on~$M$ and $(m+2-\ell){\rm scal}^{\rm W}>0$ at some point $p\in M$.
\item[$(3)$] $\mu_q<-\frac{m\ell}{m+2}$, $\rho_\theta$ semidefinite on~$M$ and $(m+2+\ell){\rm scal}^{\rm W}>0$ at some point $p\in M$.
\end{enumerate}
\item[$(b)$] If $|\ell|>m+2$ and $\rho_\theta\not\equiv 0$ is negative semidefinite then
any harmonic spinor is a section of the extremal bundles $\Sigma^m\oplus \Sigma^{-m}$.
\item[$(c)$] If $|\ell|<m+2$ and $\rho_\theta\not\equiv 0$ is positive semidefinite then
any harmonic spinor is a section of the extremal bundles $\Sigma^m\oplus \Sigma^{-m}$.
\end{enumerate}
\end{Proposition}

\begin{proof} Let $\phi=\phi_{\mu_q}\not\equiv 0$ be a spinor with $\Theta\phi=(m-2q)\phi$. We put
\begin{gather} \label{KTerm}
Q^q = - \frac{{\rm i}}{2}\left( \frac{\ell}{m+2}+\frac{\mu_q}{m}\right)\rho_\theta + \left(1+\frac{\ell\mu_q}{m(m+2)}\right)\frac{{\rm scal}^{\rm W}}{4}
\end{gather}
and \[A(\phi) = \int_M \big\langle Q^q\phi,\phi\big\rangle \operatorname{vol}_\theta.\]
The Schr\"odinger--Lichnerowicz-type formula (\ref{formtheo}) gives
\begin{equation}\label{sc2}
\|D_\theta\phi\|^2 = \frac{2q}{m}\|\nabla_{10}\phi\|^2 + \frac{2(m-q)}{m}\|\nabla_{01}\phi\|^2 + A(\phi).
\end{equation}

If the eigenvalues of $\rho_\theta$ have no different signs, we have by Cauchy--Schwarz inequality
\[
\left|\left\langle \frac{{\rm i}\rho_\theta}{2}\cdot\phi,\phi\right\rangle\right| \leq \frac{{\rm scal}^{\rm W}}{4}|\phi|^2.
\]
Hence,
\[
A(\phi) \geq \frac{(m-\mu_q)(m+2-\ell)}{m(m+2)} \int_M \frac{{\rm scal}^{\rm W}}{4}|\phi|^2\operatorname{vol}_\theta
\]
for $m\ell+(m+2)\mu_q\geq 0$, and
\[
A(\phi) \geq \frac{(m+\mu_q)(m+2+\ell)}{m(m+2)} \int_M \frac{{\rm scal}^{\rm W}}{4}|\phi|^2\operatorname{vol}_\theta
\]
for $m\ell+(m+2)\mu_q\leq 0$.

We assume that $\mu_q\neq \pm m$ is non-extremal. Then $A(\phi)$ is obviously positive for the three cases of part~(a) of the proposition.
In particular, the right hand side of~(\ref{sc2}) is always positive. This shows that no harmonic spinors exist in these three cases.

If $|\ell|>m+2$ and $\rho_\theta\not\equiv 0$ is negative semidefinite then either condition (2) or (3) of part (a) is satisfied
for any non-extremal $\Theta$-eigenvalue $\mu_q\neq \pm m$. If
$|\ell|<m+2$ and $\rho_\theta\not\equiv 0$ is positive semidefinite, one of the three conditions of part (a) is satisfied
for any non-extremal $\Theta$-eigenvalue $\mu_q\neq \pm m$.
\end{proof}

Note that we have no vanishing results for harmonic spinors in the extremal bundles~$\Sigma^{-m}$ or~$\Sigma^{m}$. Such spinors
are simply {\em holomorphic} or {\em antiholomorphic}, respectively. So far we also have no vanishing results for the canonical and anticanonical ${\rm spin}^\CC$ structures when $\ell=\pm (m+2)$. However, there are vanishing results in these cases (see~\cite{Tan} and Section~\ref{Sec10}).

\begin{Example}
Any pseudo-Einstein spin manifold $M$ (i.e., $\ell=0$) admits parallel spinors in the extremal bundles, no matter of the sign of the
Webster scalar curvature (see~\cite{LeiN2}). For ${\rm scal}^{\rm W}>0$ these are the only harmonic spinors on closed~$M$.
\end{Example}

Let us consider the CR-covariant components $\mathcal{D}_\ell$ of the Kohn--Dirac operator.
The formal adjoint of $\mathcal{D}_\ell$, $\ell\in\{-m,-m+2,\dots,m-2,m\}$, is the restriction of $D_\theta$ to the image of $\mathcal{D}_\ell$, i.e.,
\[
\mathcal{D}_\ell^*=D_\theta\colon \ \operatorname{Im}(\mathcal{D}_\ell) \to \Gamma\big(\Sigma^{-\ell}\big).
\]
Then the Schr\"odinger--Lichnerowicz-type formula~(\ref{formtheo}) for $\mathcal{D}_\ell$
is expressed by
\begin{gather}
\mathcal{D}_\ell^*\mathcal{D}_\ell = \frac{m+\ell}{m}\nabla^*_{10}\nabla_{10} + \frac{m-\ell}{m}\nabla^*_{01}\nabla_{01}
 + \frac{{\rm i}\ell\rho_\theta}{m(m+2)}
+\left(1-\frac{\ell^2}{m(m+2)}\right)\frac{{\rm scal}^{\rm W}}{4}.\label{SLformula}
\end{gather}
Especially, for $\ell=0$, we have
\[
\mathcal{D}_0^*\mathcal{D}_0 = \Delta^{\rm tr} + \frac{{\rm scal}^{\rm W}}{4}.
\]
The latter formula {\em looks} like the classical Schr\"odinger--Lichnerowicz formula of Riemannian geometry.
This immediately shows that in the spin case $h_{\frac{m}{2}}(\alpha)>0$ (i.e., a harmonic spinors of weight $\ell=0$ exists)
poses an obstruction to the existence of any adapted pseudo-Hermitian structure on the CR manifold $(M,H(M),J)$ with positive Webster scalar ${\rm scal}^{\rm W}>0$.
We give a more general version of this statement in terms of Kohn--Rossi cohomology in Corollary~\ref{ObsKR}.
Similarly, (\ref{SLformula}) implies that the CR invariants $h_{\frac{m+\ell}{2}}(\alpha)>0$, $|\ell|< m$, are obstructions
to the positivity of the Ricci form $\rho_\theta>0$.

\begin{Example} In Section~\ref{Sec12} we construct closed CR manifolds over {\em hyperK\"ahler manifolds} which admit harmonic spinors
of weight $\ell=0$. Such CR manifolds admit no adapted pseudo-Hermitian structure~$\theta$ of positive Webster scalar curvature.
\end{Example}
\begin{Example}
 There exist compact quotients of the {\em Heisenberg group}, which are strictly pseudoconvex and spin with harmonic spinors of weight $\ell=0$ (see~\cite{Stadt}).
\end{Example}

\section{Harmonic theory for the Kohn--Rossi complex} \label{Sec9}

We briefly review here the {\em Kohn--Rossi complex} \cite{KR} over
CR manifolds, twisted with some CR vector bundle $E$.
With respect to a pseudo-Hermitian form we construct the {\em Kohn Laplacian}~$\square_E$. Even though
$\square_E$ is not an elliptic operator,
there is a well behaving {\em harmonic theory}, similar to {\em Hodge theory}.
In particular, {\em Kohn--Rossi cohomology groups} are finite
and cohomology classes admit unique harmonic representatives over closed manifolds. This theory is due to J.J.~Kohn (see \cite{FK, Koh}). Our exposition of the topic follows~\cite{Tan} by N.~Tanaka.

Let $\big(M^{2m+1}, H(M), J\big)$ be a closed manifold equipped with a strictly pseudoconvex CR structure of hypersurface type and CR dimension $m\geq 1$.
With respect to the complex structure $J$ we have the decomposition $H(M)\otimes \CC=T_{10}\oplus T_{01}$
of the Levi distribution.
We define complex differential forms of degree $(p,q)$ on $H(M)$ by
\[
\Lambda^{p,q}(H(M)) := \Lambda^p T_{10}^* \otimes \Lambda^q T_{01}^*.
\]
Then
\[
\Lambda^r(H(M))\otimes\CC = \bigoplus_{p+q=r} \Lambda^{p,q}(H(M)).
\]
(Note that $(p,q)$-forms on $H(M)$ are not complex differentials form on $M$.)

We are interested in the bundles $\Lambda^{0,q}(H(M))=\Lambda^q T_{01}^*$ of $(0,q)$-forms.
The corresponding spaces of smooth sections over $M$ are denoted by $\mathcal{C}^q(M)$, $q=0,\dots, m$.
There exist {\em tangential Cauchy--Riemann operators}
\[
\bar{\partial}_b\colon \ \mathcal{C}^q(M) \to \mathcal{C}^{q+1}(M),\qquad q\in\{0,\dots,m\}.
\]
These differential operators are by construction CR invariants and
the sequence
\[
0\longrightarrow \mathcal{C}^0(M) \stackrel{\bar{\partial}_b}{\longrightarrow} \mathcal{C}^1(M) \stackrel{\bar{\partial}_b}{\longrightarrow} \cdots
 \stackrel{\bar{\partial}_b}{\longrightarrow} \mathcal{C}^{m-1}(M) \stackrel{\bar{\partial}_b}{\longrightarrow} \mathcal{C}^m(M)\longrightarrow 0
\]
is called {\em Kohn--Rossi complex}. Its cohomology groups are denoted by $H^{0,q}(M)$, $q\geq 0$.

More generally, let us consider a complex vector bundle~$E$ over~$M$. We assume that $E$ is equipped with some Cauchy--Riemann operator
\[ \bar{\partial}_E\colon \ \Gamma(E) \to \Gamma(E\otimes T_{01}^*),\]
i.e., $\bar{\partial}_E$ satisfies
\begin{gather*} \bar{\partial}_E(fu)(X_{01})=X_{01}(f)\cdot u+ f\cdot \big(\bar{\partial}_Eu\big)(X_{01}),\\
(\bar{\partial}_Eu)([X_{01},Y_{01}])=\bar{\partial}_E\big(\bar{\partial}_Eu(X_{01})\big)(Y_{01})- \bar{\partial}_E\big(\bar{\partial}_Eu(Y_{01})\big)(X_{01})
\end{gather*}
for any smooth $\CC$-valued function $f$ and sections $X_{01}$, $Y_{01}$ in $T_{01}$. We call $\big(E,\bar{\partial}_E\big)$ a {\em CR vector bundle} over $M$. Smooth sections $u$ of $E$ with $\bar{\partial}_Eu=0$ are {\em holomorphic sections} (see~\cite{Tan}).

Furthermore, for some CR vector bundle $E$ over $M$, we set $C^q(M,E)=\Lambda^qT_{01}^*\otimes E$ and
$\mathcal{C}^q(M,E)=\Gamma(C^q(M,E))$ for smooth sections. The {\em holomorphic structure} $\bar{\partial}_E$ on $E$ extends to Cauchy--Riemann operators
\[ \bar{\partial}_E\colon \ \mathcal{C}^q(M,E) \to \mathcal{C}^{q+1}(M,E)\]
for any $q\in\{0,\dots,m\}$. This is by construction a twisted complex $\big( \mathcal{C}^q(M,E), \bar{\partial}_E\big)$, and
we denote the corresponding cohomology groups by $H^q(M,E)$, $q\in\{0,\dots,m\}$.
For~$E$ the trivial line bundle over~$M$, these are the Kohn--Rossi cohomology groups $H^{0,q}(M)$ (see~\cite{Tan}).

Let us assume now that a pseudo-Hermitian form $\theta$ is given on $M$ and
that the CR vector bundle $E\to M$ is equipped with a Hermitian inner product
$\langle\cdot,\cdot\rangle_E$. In this setting we have a~direct sum decomposition
\[
T(M)\otimes\CC = T_{10}\oplus T_{01}\oplus \RR T_\theta,
\]
which gives rise to a unique identification of $C^q(M,E)$ with a subbundle of $\Lambda^q(T^*(M))\otimes E$.
And there exists
a {\em canonical connection} $D\colon \Gamma(E)\to\Gamma(T^*(M)\otimes E)$ compatible with $\langle\cdot,\cdot\rangle_E$
and related to the Cauchy--Riemann operator by $D_{X_{01}}u=\bar{\partial}_Eu(X_{01})$, $X_{01}\in T_{01}$,
for any $u\in\Gamma(E)$.
Together with the Webster--Tanaka connection $\nabla^{\rm W}$ we obtain covariant derivatives
\[
D\colon \ \Gamma\big(\Lambda^q(M)\otimes E\big) \to \Gamma\big(\Lambda^{q+1}(M)\otimes E\big),\qquad q\in\{0,\dots,m\},
\]
and \looseness=-1 with respect to a local frame $(E_1,\dots,E_m)$ of $T_{10}$ the Cauchy--Riemann operators are given by
\[
\bar{\partial}_Eu = \sum_{\alpha=1}^m \overline{E_\alpha}^*\wedge D_{\overline{E}_\alpha}u
\]
for $u\in\mathcal{C}^q(M,E)$.

Moreover, for any $q\in\{0,\dots,m\}$, the vector bundle $C^q(M,E)$ is equipped with a Hermitian inner product,
which gives rise via $\operatorname{vol}_\theta$ to an $L_2$-inner product on $\mathcal{C}^q(M,E)$.
This allows the construction of a formally adjoint differential operator
\[
\bar{\partial}^*_E\colon \ \mathcal{C}^{q+1}(M,E) \to \mathcal{C}^{q}(M,E)
\]
to $\bar{\partial}_E$.
With respect to a local frame $(E_1,\dots,E_m)$ the operator $\bar{\partial}^*_E$ is given by
\[
\bar{\partial}^*_Eu = -\sum_{\alpha=1}^m \iota_{\overline{E_\alpha}} D_{E_\alpha}u
\]
for $u\in\mathcal{C}^{q+1}(M,E)$.

Finally, we can construct the {\em Kohn Laplacian}
\[
\square_E := \bar{\partial}^*_E\bar{\partial}_E+ \bar{\partial}_E\bar{\partial}^*_E\colon \ \mathcal{C}^{q}(M,E) \to \mathcal{C}^{q}(M,E),\qquad q\in\{0,\dots,m\},
\]
with respect to $\theta$ on $M$.
This is a $2$nd order differential operator, which is formally self-adjoint with respect to $(\cdot,\cdot)_{L_2}$ on~$\mathcal{C}^{q}(M,E)$.
Due to results of Kohn the operator $\square_E$ is sub- and hypoelliptic.
We put
\[
\mathcal{H}^q(M,E) := \big\{u\in\mathcal{C}^q(M,E)\,|\, \square_Eu=0\big\}
\]
for the space of harmonic $(0,q)$-forms. Since $\square_E$ is formally self-adjoint the harmonic equation $\square_Eu=0$ is equivalent to
$\bar{\partial}_Eu=\bar{\partial}^*_Eu=0$ on~$M$.
It is known that $\mathcal{H}^q(M,E)$ is finite dimensional for any $q\in\{1,\dots,m-1\}$. Moreover,
every class in the Kohn--Rossi cohomology group $H^q(M,E)$ admits a unique harmonic representative, i.e.,
\[
H^q(M,E) \cong \mathcal{H}^q(M,E).
\]
In particular, the Kohn--Rossi cohomology groups $H^q(M,E)$ are finite dimensional for $q\in\{1,\dots,m-1\}$. The
groups $H^0(M,E)$ and $H^m(M,E)$ are infinite dimensional, in general. However, we still have
$H^0(M,E)\cong \mathcal{H}^0(M,E)$ and $H^m(M,E) \cong \mathcal{H}^m(M,E)$ for any $m\geq 2$ (cf.~\cite{Tan}).
In case $m=1$ and $M\subseteq\CC^2$ is embedded as CR manifold $H^0(M,E) \cong \mathcal{H}^0(M,E)$ and $H^1(M,E) \cong \mathcal{H}^1(M,E)$ are certainly true as well.

\section{Vanishing theorems for twisted Kohn--Rossi cohomology} \label{Sec10}

The harmonic theory of the previous section fits well to our discussion of Kohn--Dirac operators and harmonic spinors.
In fact, the square $D_\theta^2$ of the Kohn--Dirac operator has a natural interpretation as
Kohn Laplacian~$\square_E$ if we only make the appropriate choice for the CR line bundle~$E$ (see~\cite{Pet}). This justifies the name for $D_\theta$
and gives rise via the Schr\"odinger--Lichnerowicz-type formula to vanishing results for twisted Kohn--Rossi cohomology (see in \cite[Section~II, \S~7]{Tan}; cf.~\cite{Lee}).

Let $\big(M^{2m+1}, H(M), J\big)$ be a closed manifold equipped with strictly pseudoconvex CR structure of hypersurface type and CR dimension $m\geq 1$.
We fix a pseudo-Hermitian form $\theta$ on~$M$ with ${\rm spin}^\CC$ structure of weight $\ell\in\ZZ$.
The corresponding spinor bundle $\Sigma(H(M))\to M$ decomposes into
\[%\label{deco}
\Sigma(H(M)) = \bigoplus_{q=0}^m \Sigma^{\mu_q}(H(M)).
\]
The Kohn--Dirac operator is given by $D_\theta=D_++D_-$. In particular, we have the spinorial complex~(\ref{spincpx}) $\big(\Gamma\big(\Sigma^{\mu_q}\big),D_+\big)$
with cohomology groups, which we denote by $S^q(M)$, $q=0,\dots, m$ (cf. the notion of spinorial cohomology in~\cite{Mic}).

Recall that the chosen ${\rm spin}^\CC$ structure on $(M,\theta)$ is
uniquely determined by some complex line bundle $E(\alpha)\to M$, $\alpha\in H^2(M,\ZZ)$, which is a square root of $\mathcal{K}\otimes L$,
$L=\mathcal{E}(\ell)$ the determinant bundle. Note that we can use the Webster--Tanaka connection $\nabla^{\rm W}$ to define a holomorphic
structure on $E(\alpha)$ through $\bar{\partial}_{E(\alpha)}\eta(X_{01}):=\nabla^{\rm W}_{X_{01}}\eta$, $X_{01}\in T_{01}$, for $\eta\in\Gamma(E(\alpha))$.

Studying the spinor module $\Sigma$ with Clifford multiplication $c$ shows that the spinor bundle $\Sigma(H(M))\to M$ is isomorphic to
\[
\bigoplus_{q=0}^m \Lambda^{q}T_{01}^* \otimes \Sigma^{\mu_0}(H(M)).
\]
Moreover, the factor $\Sigma^{\mu_0}(H(M))$ is isomorphic to the line bundle $E(\alpha)$. In fact, we have
$\Sigma^{\mu_q}(H(M))\cong \Lambda^{q}T_{01}^*\otimes E(\alpha)$ of rank
$\operatorname{rk}_\CC=\left( \begin{smallmatrix} m\\ q\end{smallmatrix}\right)$.
Hence, the identifications{\samepage
\begin{equation}\label{key}
\Gamma(\Sigma^{\mu_q}) \cong \mathcal{C}^q(M,E(\alpha)),\qquad q=0,\dots,m,
\end{equation}
with the chain groups of the Kohn--Rossi complex, twisted by $E(\alpha)$.}

Examining the Clifford multiplication shows that
the operator $\frac{1}{\sqrt{2}}D_+$ corresponds via~(\ref{key}) to the Cauchy--Riemann operator $\bar{\partial}_{E(\alpha)}$.
In particular, we have
\[
S^q(M) \cong H^q(M,E(\alpha)),\qquad q\in\{0,\dots,m\},
\]
for the cohomology groups of the spinorial complex.

Moreover, the formal adjoint $\frac{1}{\sqrt{2}}D_-$ corresponds to $\bar{\partial}^*_{E(\alpha)}$. Hence, via~(\ref{key})
we have
\[
D_\theta = \sqrt{2}\cdot \big(\bar{\partial}_{E(\alpha)} + \bar{\partial}^*_{E(\alpha)}\big)
\]
for \looseness=-1 the Kohn--Dirac operator and $D^2_\theta =2\square_{E(\alpha)}$ for the square (cf.~\cite{Pet}). This shows that harmonic spinors with
$\Theta$-eigenvalue $\mu_q=m-2q$, $q=0,\dots,m$, are in $1$-to-$1$-correspondence with harmonic $(0,q)$-forms with values in~$E(\alpha)$.
In particular, with results of Section \ref{Sec9} we can interpret harmonic spinors on closed~$M$ as representatives of
twisted Kohn--Rossi cohomology classes.

\begin{Theorem} \label{SKR} Let $\big(M^{2m+1}, H(M),J\big)$, $m\geq 2$, be a strictly pseudoconvex and closed CR manifold
with pseudo-Hermitian form $\theta$ and ${\rm spin}^\CC$ structure, determined by $\alpha\in H^2(M,\ZZ)$.
\begin{enumerate}\itemsep=0pt
\item[$1.$] For the space of harmonic spinors to the eigenvalue $\mu_q$, $q=0,\dots,m$, we have
\[
\mathcal{H}^q(\alpha) \cong \mathcal{H}^q(M,E(\alpha)) \cong H^q(M,E(\alpha)) \cong S^q(M).\]
\item[$2.$] The space $\mathcal{H}^q(\alpha)$ of harmonic spinors
on $M$ is finite dimensional for any $q\in\{1,\dots,m-1\}$.
\end{enumerate}
\end{Theorem}

\begin{Remark} The cohomology groups $H^q(M,E(\alpha))$, $q\in\{0,\dots,m\}$, of the twisted Kohn--Rossi complex are invariant objects of the underlying CR structure on $M$, whereas the construction of the spaces of harmonic spinors $\mathcal{H}^q(\alpha)$ and harmonic $(0,q)$-forms $\mathcal{H}^q(M,E(\alpha))$ depends on the pseudo-Hermitian structure~$\theta$. It is only for $\mu_q=-\ell$ that we have seen in Section~\ref{Sec7} that harmonic spinors in $\mathcal{H}^{\frac{m+\ell}{2}}(\alpha)$ are solutions of a CR invariant equation.

Theorem~\ref{SKR} shows now that elements
in $\mathcal{H}^q(\alpha)$, $q\in\{0,\dots,m\}$, can be identified for different pseudo-Hermitian forms $\theta$ and $\tilde{\theta}$ via the corresponding
Kohn--Rossi cohomology groups. In particular, all the dimensions $h_q(\alpha)$ of the spaces $\mathcal{H}^q(\alpha)$, $q\in\{1,\dots,m-1\}$, are CR invariant numbers.
\end{Remark}

Recall that a spin structure for the Levi distribution $H(M)$ on $M$ is given by some square root $E(\alpha)$
of the canonical bundle $\mathcal{K}\to M$. We denote the chosen square root by~$\sqrt{\mathcal{K}}$.
On the other hand, the canonical ${\rm spin}^\CC$ structure on~$M$ is given by the trivial line bundle $E(\alpha)=M\times\CC$.
In this case $\ell=m+2$. We have the following vanishing results for Kohn--Rossi cohomology.

\begin{Theorem} \label{VanKR} Let $\big(M^{2m+1}, H(M),J\big)$, $m\geq 2$, be a strictly pseudoconvex and closed CR manifold
with pseudo-Hermitian form $\theta$ and ${\rm spin}^\CC$ structure of weight~$\ell$ determined by $\alpha\in H^2(M,\ZZ)$.
\begin{enumerate}\itemsep=0pt
\item[$1.$] If $\rho_\theta\not\equiv 0$ is negative semidefinite, $|\ell|>m+2$ and $q\in\{1,\dots,m-1\}$, then
\[
H^q(M,E(\alpha)) = \{0\}
\]
for the $q$th $(\alpha$-twisted$)$ Kohn--Rossi cohomology group.
\item[$2.$] If $\rho_\theta\not\equiv 0$ is positive semidefinite, $|\ell|<m+2$ and $q\in\{1,\dots,m-1\}$, then
\[
H^q(M,E(\alpha)) = \{0\}.
\]
\item[$3.$] If $\rho_\theta>0$ is positive definite, $E(\alpha)=M\times\CC$ and $q\in\{1,\dots,m-1\}$, then
\[
H^{0,q}(M) = \{0\}
\]
for the $q$th Kohn--Rossi group.
\item[$4.$] If ${\rm scal}^{\rm W}\not\equiv 0$ is non-negative on a closed CR spin manifold~$M$ of even CR dimension~$m$, then
\[
H^{\frac{m}{2}}\big(M,\sqrt{\mathcal{K}}\big) = \{0\}.
\]
\end{enumerate}
\end{Theorem}

\begin{proof} Part (1) and (2) of Theorem \ref{VanKR} follow immediately from Proposition~\ref{vani} via the identifications in Theorem~\ref{SKR}.
Part (3) is the statement of Proposition~7.4 on p.~62 in~\cite{Tan} for Kohn--Rossi cohomology.
We reprove this result here.

The curvature term (\ref{KTerm}) decomposes into two summands as follows:
\begin{align*}
Q^q & = - \frac{{\rm i}}{2}\left( \frac{\ell}{m+2}+\frac{\mu_q}{m}\right)\rho_\theta + \left(1+\frac{\ell\mu_q}{m(m+2)}\right)\frac{{\rm scal}^{\rm W}}{4}\\
&= \left( 1+ \frac{\mu_q}{m} \right)\left( -\frac{{\rm i}}{2}\rho_\theta +\frac{{\rm scal}^{\rm W}}{4} \right) -
\frac{{\rm i}(\ell-m-2)}{2(m+2)}\left( \rho_\theta-\frac{{\rm scal}^{\rm W}\cdot {\rm d}\theta}{4m}\right)\\
&=: \frac{2(m-q)}{m}R_* + K.
\end{align*}
The second summand $K$ vanishes for $\ell=m+2$. In any case we have $\operatorname{tr}_\theta K=0$.

Via (\ref{key}) Clifford multiplication on $(0,q)$-forms $u\in \Lambda^qT^*_{01}$ is given by
\[
X\cdot u = \sqrt{2}\cdot\big( X^*_{01}\wedge u - \iota_{X_{01}}u \big)
\]
for any $X\in H(M)$ (see, e.g., \cite{Morg}~and~\cite{Pet}). Then a short computations shows that $R_*=-\frac{{\rm i}}{2}\rho_\theta +\frac{{\rm scal}^{\rm W}}{4}$ acts
on $u\in C^q(M,E(\alpha))$ by
\[
(R_*u)\big(\overline{X}_1,\dots,\overline{X}_q\big) = \sum_{\alpha=1}^q u\big(\overline{X}_1,\dots,\rho_\theta\big(\overline{X}_\alpha\big),\dots, \overline{X}_q\big)
\]
for any $\overline{X}_1,\dots,\overline{X}_q\in \Lambda^qT^*_{01}$, i.e., $R_*$ is the {\em Ricci operator} on $C^q(M,E(\alpha))$.

On the other hand, the action of $K$ is induced by the curvature of the {\em canonical connection} on~$E(\alpha)$ (which differs from the Webster--Tanaka connection, in general). Thus, $Q^q=\frac{2(m-q)}{m}R_*+K$ is exactly the curvature term of the Weitzenb\"ock formula in Proposition~5.1 on p.~47 of~\cite{Tan}.

In particular, for the case of the canonical ${\rm spin}^\CC$ structure $\alpha=0$, we have $\ell=m+2$ and $K=0$.
If $\rho_\theta> 0$ is positive definite on~$M$, $R_*$ is positive definite as well. Application of the Weitzenb\"ock formula for
$\mathcal{C}^q(M,E(\alpha))$ shows that there are no harmonic forms.

Finally, in the spin case $\ell=0$ with even $m$, we have $Q^{\frac{m}{2}}=\frac{{\rm scal}^{\rm W}}{4}$ and~(\ref{formtheo}) shows that there are no harmonic spinors. Hence, no harmonic forms in $\mathcal{C}^{\frac{m}{2}}\big(M,\sqrt\mathcal{{K}}\big)$.
\end{proof}

On the other hand, non-trivial Kohn--Rossi groups pose obstructions to positive Webster curvature on the underlying CR manifold.
We put
\[
\hat{q}=\hat{q}(m,\ell) := \frac{m(m+\ell+2)}{2(m+2)}
\]
and highlight the following result, which resembles the classical obstruction for positive scalar curvature on K\"ahler manifolds
(cf.~\cite{Hitch, Lich}).

\begin{Corollary} \label{ObsKR} Let $\big(M^{2m+1}, H(M),J\big)$, $m\geq 2$, be a strictly pseudoconvex and closed CR manifold with ${\rm spin}^\CC$ structure
of weight $\ell\in\ZZ$ to $\alpha\in H^2(M,\ZZ)$. If
\[
|\ell|<m+2,\qquad \hat{q}\in\ZZ\qquad\mbox{and}\qquad H^{\hat{q}}(M,E(\alpha)) \neq \{0\},
\]
then $M$ admits no adapted pseudo-Hermitian structure $\theta$ of positive Webster scalar curvature ${\rm scal}^{\rm W}>0$.
\end{Corollary}

\begin{proof} For $\hat{q}\in\ZZ$, we have $\mu_{\hat{q}}=-\frac{m\ell}{m+2}\in\ZZ$ and $Q^{\hat{q}}=\big(1-\frac{\ell^2}{(m+2)^2}\big)\frac{{\rm scal}^{\rm W}}{4}$.
If $|\ell|<m+2$, then $\hat{q}\in\{1,\dots,m-1\}$ and the functions $Q^{\hat{q}}$ and ${\rm scal}^{\rm W}$ have the same sign.
The non-vanishing of $H^{\hat{q}}(M,E(\alpha))$ implies the existence of
a harmonic spinor in $\Gamma(\Sigma^{\mu_{\hat{q}}})$. This is impossible by~(\ref{formtheo}) when ${\rm scal}^{\rm W}>0$.
\end{proof}

\begin{Example} If $M$ is a closed and strictly pseudoconvex CR spin manifold of even CR dimension $m \geq 2$, then
the cohomology group
\[
H^{\frac{m}{2}}\big(M,\sqrt{\mathcal{K}}\big)
\]
poses an obstruction to ${\rm scal}^{\rm W}>0$. We give some concrete example of this case in Section \ref{Sec12} (cf.\ also Corollary~\ref{cordem}).
\end{Example}

\begin{Example} If $m=4$ and $\ell=-3$, then $\hat{q}=1$ and $H^1(M,E(\alpha))$ poses an obstruction to
${\rm scal}^{\rm W}>0$.
\end{Example}

\begin{Example} Let $\big(M^{2m+1},\theta\big)$, $m\geq 2$, be some pseudo-Einstein space with $\operatorname{Ric}^{\rm W}(T)=0$.
Then ${\rm scal}^{\rm W}$ is constant on~$M$. E.g., this happens when the Webster torsion $\tau$ is parallel,
or when the characteristic vector $T_\theta$ is a transverse symmetry of the underlying CR structure
(cf.\ Section~\ref{Sec11}).

We assume a ${\rm spin}^\CC$ structure of weight $\ell$ and $\rho_\theta\not\equiv 0$.
Then, either $\rho_\theta>0$ and, for $|\ell|\leq m+2$ and $1<q<m$,
we have $H^q(M,E(\alpha))=\{0\}$ (by Theorem~\ref{VanKR}).
E.g., {\em Einstein--Sasakian} manifolds of Riemannian signature give rise to such pseudo-Hermitian structures $\theta$
(see, e.g.,~\cite{Blair} for the notion of Sasakian structures).
In Section~\ref{Sec12} we construct {\em regular} Einstein--Sasakian manifolds from K\"ahler geometry.

In the other case $\rho_\theta<0$ and $H^q(M,E(\alpha))=\{0\}$ for $|\ell|> m+2$ and any $1<q<m$.
E.g., {\em Einstein--Sasakian} manifolds of Lorentzian signature fit to this situation.
\end{Example}

\section{Cohomology of regular, torsion-free CR manifolds} \label{Sec11}

Let $\big(M^{2m+1},H(M),J\big)$, $m\geq 2$, be strictly pseudoconvex with pseudo-Hermitian form~$\theta$. We call the characteristic vector
$T_\theta$ {\em regular} if all its integral curves are $1$-dimensional submanifolds of~$M$ and the corresponding leaf space $N$ is a~smooth manifold
of dimension $2m$ with smooth projection $\pi\colon M\to N$. If, in addition, $\theta$ has vanishing Webster torsion $\tau=0$ then $T_\theta$ is an infinitesimal
automorphism of the underlying CR structure and a {\em Killing vector} for the {\em Webster metric} $g^{\rm W}=g_\theta+\theta\circ\theta$ on $M$.
Such a vector $T_\theta$ is called {\em transverse symmetry}. (Here we understand {\em transverse} to $H(M)$.)
It is
straightforward to see that in this case the pseudo-Hermitian structure
 $(H(M),J,\theta)$ projects to a K\"ahler structure on the leaf space~$N$. Note that if~$M$ is closed then $\pi\colon M\to N$ is a circle fiber bundle. In this case we call $\pi\colon (M,\theta)\to N$ a~(regular, torsion-free) {\em CR circle bundle} of complex dimension $m\geq 2$.

\subsection{Holomorphic cohomology and vanishing theorems}
Let us consider the underlying leaf space $N$ with K\"ahler metric $h$, complex structure $J$ and fundamental form $\omega$.
As for any complex manifold, we have the $(p,q)$-forms $\Lambda^{p,q}(N)$ and the Cauchy--Riemann operators
\[
\bar{\partial}\colon \ \Gamma\big(\Lambda^{p,q}(N)\big)\to \Gamma\big(\Lambda^{p,q+1}(N)\big),
\]
which in turn give rise to the {\em Dolbeault cohomology groups} $H^{p,q}(N)$, $p,q\geq 0$.
In the following we are interested in the cohomology groups $H^{0,q}(N)$, $q\geq 0$.
In fact, more generally, let $E'\to N$ be some holomorphic vector bundle and~$\mathcal{O}(E')$ the sheaf of local holomorphic sections
of~$E'$. Then we have the $q$th cohomology group $H^q(N,\mathcal{O}(E'))$ of the sheaf~$\mathcal{O}(E')$, which is, by
{\em Dolbeault's theorem}, isomorphic to $H^{0,q}(N,E')$, $q\geq 0$.

Now let $E'$ be some complex line bundle over the K\"ahler manifold~$N$. We assume that~$E'$ is a root of some power of
the anticanonical line bundle $\mathcal{K}'^{-1}\to N$. This ensures that~$E'$ is equipped with a holomorphic structure and Hermitian inner product,
both compatible with the Levi-Civita connection of the K\"ahler metric.
The pullback $E=\pi^*E'$ is a line bundle over~$M$ with Hermitian inner product and the Webster--Tanaka connection induces some holomorphic structure on~$E$
as well.
Then any smooth section $u'\in\Gamma(E')$ lifts to some smooth section $u=\pi^*u'$ of $E\to M$. By construction, the Lie derivative
$\mathfrak{L}_Tu$ of the lift in characteristic direction vanishes identically on $M$. In fact, any smooth section $u\in\Gamma(E)$ with $\mathfrak{L}_Tu=0$ is the pullback of some unique section
$u'$ in $E'\to N$. We call such sections in $E\to M$ {\em projectable}.
More generally, for $q\geq 0$, we have the subspaces \[\mathcal{C}_{(0)}^q(M,E) \subseteq \mathcal{C}^q(M,E)\] of projectable $(0,q)$-forms in the chain groups
of the Kohn--Rossi complex with values in the line bundle $E$. These subgroups are
naturally identified with the $(0,q)$-forms
$\Gamma\big(\Lambda^{0,q}\otimes E'\big)$ with values in~$E'$ over the K\"ahler manifold~$N$.

Since the Webster torsion $\tau_\theta$ vanishes on~$M$, the holomorphic structures on $E'$ and on its pullback~$E$ are compatible:
$\bar{\partial}_E\pi^*v'=\pi^*\bar{\partial}_{E'}v'$ for any $(0,q)$-form~$v'$ on~$N$.
Also $\bar{\partial}^*_E\pi^*v'=\pi^*\bar{\partial}^*_{E'}v'$ is true.
Now let $\gamma\in H^{0,q}(N,E')$ be a class in the $q$th Dolbeault group. By classical {\em Hodge theory}~$\gamma$ is
uniquely represented by some harmonic $(0,q)$-form~$u'$ with values in~$E'$. In general,
the lift $u=\pi^*u'$ of some non-trivial harmonic~$u'$ is a non-trivial element of $\mathcal{C}_{(0)}^q(M,E)$, which is
harmonic with respect to the Kohn Laplacian~$\square_E$. Thus, we have an inclusion
\[
\pi^*\colon \ \mathcal{H}^q(N,E') \hookrightarrow \mathcal{H}^q(M,E)
\]
of spaces of harmonic forms. This again gives rise to a natural inclusion
\[
\pi^*\colon \ H^q(N,\mathcal{O}(E')) \hookrightarrow H^q(M,E)
\]
of the holomorphic cohomology group $H^q(N,\mathcal{O}(E'))$ into the Kohn--Rossi group $H^q(M,E)$ for any $q\geq 0$.
We denote the image of this inclusion by $H_{(0)}^q(M,E)$. By construction, classes in $H_{(0)}^q(M,E)$ are represented
by projectable harmonic $(0,q)$-forms on $M$, i.e., by elements of~$\mathcal{H}_{(0)}^q(M,E)$.

Let us assume now that the given holomorphic line bundle $E'\to N$ is a square root of $\mathcal{K}'\otimes L'$, where $L'\to N$
is a line bundle of weight $\ell\in\ZZ$, i.e., $L'^{m+2}$ is the $\ell$th power of the anticanonical bundle $\mathcal{K}'^{-1}$ over~$N$. Then $E'$ determines a ${\rm spin}^\CC$ structure on $N$ with determinant bundle~$L'$ of weight~$\ell$. This lifts to a ${\rm spin}^\CC$ structure on $(M,\theta)$ with determinant bundle $L=\pi^*L'$ of weight~$\ell$. Of course, the corresponding spinor bundle $\Sigma'\to N$ pulls back to the spinor bundle
$\Sigma$ over $(M,\theta)$ and harmonic spinors in $\Sigma'$ lift to harmonic spinors in $\Sigma\to M$.

\begin{Theorem} \label{RegVan} Let $\pi\colon (M,\theta)\to N$ be some CR circle bundle of complex dimension $m\geq 2$ with ${\rm spin}^\CC$ structure of weight $|\ell|< m+2$
$($determined by some line bundle $E=\pi^*E'\to M)$ and assume $\hat{q}=\frac{m(m+\ell+2)}{2(m+2)}\in\ZZ$.
\begin{enumerate}\itemsep=0pt
\item[$(a)$] If the scalar curvature ${\rm scal}^h\not\equiv 0$ is non-negative on the K\"ahler manifold~$N$ then
there exist no harmonic spinors on~$M$ to the $\Theta$-eigenvalue $\mu_{\hat{q}}=-\frac{m\ell}{m+2}$ and the Kohn--Rossi cohomology group $H^{\hat{q}}(M,E)=\{0\}$
is trivial.
\item[$(b)$] If the holomorphic cohomology group $H^{\hat{q}}(N,\mathcal{O}(E'))\neq\{0\}$ is non-trivial then the strictly pseudoconvex CR manifold~$M$ admits
no adapted pseudo-Hermitian structure of positive Webster scalar curvature.
\end{enumerate}
\end{Theorem}

\begin{proof} Since $T_\theta$ is a transverse symmetry, we have $\iota_{T_\theta}R^{\rm W}=0$
for the Webster curvature operator. Hence, the pseudo-Hermitian Ricci form $\rho_\theta$ is the lift of the
Ricci $2$-form on $N$, and the Webster scalar curvature on~$M$ is the lift of the Riemannian scalar curvature on~$N$.

If ${\rm scal}^h\not\equiv0$ is non-negative on $N$, the curvature term $Q^{\hat{q}}=\big(1-\frac{\ell^2}{(m+2)^2}\big)\frac{{\rm scal}^{\rm W}}{4}\not\equiv 0$
is non-negative as well. The Schr\"odinger--Lichnerowicz-type formula implies $H^{\hat{q}}(M,E)=\{0\}$.
On the other hand, we have an inclusion of $H^{\hat{q}}(N,\mathcal{O}(E'))$ into $H^{\hat{q}}(M,E)$. Hence, if
$H^{\hat{q}}(N,\mathcal{O}(E'))$ is non-trivial then $H^{\hat{q}}(M,E)$ as well. By Corollary~\ref{ObsKR},
this obstructs the existence of some pseudo-Hermitian form with positive Webster scalar curvature.
\end{proof}

In case of spin structures we have the following special result.

\begin{Corollary} \label{cordem} Let $\pi\colon (M,\theta)\to N$ be some CR circle bundle of even complex dimension $m\geq 2$ with spin structure $\sqrt{\mathcal{K}'}\to N$.
\begin{enumerate}\itemsep=0pt
\item[$(a)$] If ${\rm scal}^h\not\equiv 0$ is non-negative on $N$, the Kohn--Rossi cohomology group
\[
H^{\frac{m}{2}}\big(M,\sqrt{\mathcal{K}}\big) = \{0\}
\]
is trivial.
\item[$(b)$]
If the holomorphic cohomology group $H^{\frac{m}{2}}\big(N,\mathcal{O}\big(\sqrt{\mathcal{K}'}\big)\big)\neq\{0\}$ is non-trivial then the strictly pseudoconvex CR manifold~$M$ admits no adapted pseudo-Hermitian structure of positive Webster scalar curvature.
\end{enumerate}
\end{Corollary}

\subsection{Shifting cohomology}

Let $\big(M^{2m+1},\theta\big)$, $m\geq 2$, be some closed pseudo-Hermitian manifold with vanishing Webster torsion $\tau_\theta=0$.
We assume now that the characteristic vector~$T$ of $\theta$
is induced from some free ${\rm U}(1)$-action on~$M$. This implies that $T$ is regular with underlying K\"ahler manifold~$N$.
In fact, here $\pi\colon (M,\theta)\to N$ is a smooth principal ${\rm U}(1)$-bundle and $\theta$ is a connection form.
We call $\pi\colon (M,\theta)\to N$ a {\em CR principal ${\rm U}(1)$-bundle}.
The corresponding holomorphic line bundle is denoted by $F'\to N$ and $F=\pi^*F'$ is its pullback to~$M$.

Let $\mathcal{C}^q(M)$ denote the chain groups of the Kohn--Rossi complex on~$M$. In~\cite{Tan} the differential operator
\begin{align*}
\mathcal{N}\colon \ \mathcal{C}^q(M)&\to \mathcal{C}^q(M),\qquad q\geq 0,\\
u&\mapsto {\rm i}\nabla^{\rm W}_Tu,
\end{align*}
is introduced. Let us denote by
\[
\mathcal{C}_{(\lambda)}^q(M) := \big\{u\in \mathcal{C}^q(M)\colon \mathcal{N}u= \lambda u\big\}
\]
the corresponding $\lambda$-eigenspace.
Since the Webster torsion $\tau_\theta=0$ vanishes, we have in general
$\nabla^{\rm W}_Tu = \mathfrak{L}_Tu$ for the Lie derivative of any chain $u\in\mathcal{C}^q(M)$. This shows that
$\mathcal{C}_{(0)}^q(M)$ are exactly the projectable $(0,q)$-forms on~$M$.

The operator $\mathcal{N}$ is self-adjoint and commutes with $\bar{\partial}$ and $\bar{\partial}^*$.
Hence, $\mathcal{N}$ acts on the harmonic spaces $\mathcal{H}^{0,q}(M)$, $q\geq 0$, as well. All eigenvalues~$\lambda$ are real.
In fact, since $\mathcal{H}^{0,q}(M)$ is finite dimensional for $0<q<m$, we have finitely many real eigenvalues~$\lambda$ and a direct sum decomposition
\[\mathcal{H}^{0,q}(M) = \bigoplus_{\lambda} \mathcal{H}_{(\lambda)}^{0,q}(M) ,
\]
which implies for the Kohn--Rossi groups:
\[
H^{0,q}(M) = \bigoplus_{\lambda} H_{(\lambda)}^{0,q}(M),\qquad 1\leq q\leq m-1 .
\]

In \cite{Tan} it is shown that for any integers $q,s\in\ZZ$ the cohomology groups
\begin{equation*}%\label{preshift}
H_{(s)}^{0,q}(M) \cong H^q\big(N,\mathcal{O}(F'^s)\big)
\end{equation*}
are naturally identified. Here $F'^s$ denotes the $s$th power of the line bundle $F'\to N$, which is associated to
the given CR principal ${\rm U}(1)$-bundle $\pi\colon (M,\theta)\to N$.

On the other hand, let $F=\pi^*F'\to M$ denote the pullback of $F'$.
Via the projection $\pi\colon M\to N$ the holomorphic cohomology group $H^q(N,\mathcal{O}(F'^s))$ is identified with
the subgroup $H_{(0)}^{q}(M,F^s)\subseteq H^{q}(M,F^s)$ of $F^s$-twisted Kohn--Rossi cohomology. This gives rise to a natural identification
\begin{equation}\label{shift}
H_{(s)}^{0,q}(M) \cong H^q_{(0)}(M,F^s) ,
\end{equation}
where the classes in $H^q_{(0)}(M,F^s)$ are uniquely represented by projectable harmonic $(0,q)$-forms
with values in $F^s$. We call (\ref{shift}) a {\em shift} in Kohn--Rossi cohomology over the CR principal ${\rm U}(1)$-bundle $\pi\colon (M,\theta)\to N$.

We apply (\ref{shift}) to our situation of ${\rm spin}^\CC$ structures.
For this we assume that the pullback $F=\pi^*F'\to M$ of the line bundle $F'\to N$ associated to $\pi\colon M\to N$ has weight $f\in\ZZ$.
Then $E:=F^s=\mathcal{E}(sf)$ determines a ${\rm spin}^\CC$ structure of weight $\ell=2sf+m+2$ on $(M,\theta)$.

\begin{Theorem} Let $\pi\colon \big(M^{2m+1},\theta\big)\to N$, $m\geq 2$, be some CR principal ${\rm U}(1)$-bundle of weight $f\in\ZZ$ and
$E:=\mathcal{E}(sf)$, $s\in\ZZ$. Then
\begin{enumerate}\itemsep=0pt
\item[$1.$] $H^q_{(0)}(M,E)=\{0\}$ for $q<m$ and $s>0$.
\item[$2.$] If $\frac{msf}{m+2}\in\{1-m,\dots,-1\}$ and ${\rm scal}^h>0$ then $H_{(s)}^{0,{\hat{q}}}(M)=\{0\}$ for $\hat{q}=m\big(1+\frac{sf}{m+2}\big)$.
\end{enumerate}
\end{Theorem}

\begin{proof} (1) Comparing the Kohn Laplacian $\square=\bar{\partial}^*\bar{\partial}+\bar{\partial}\bar{\partial}^*$ with
$\overline{\square}= \partial^*\partial+\partial\partial^*$ gives
$ \overline{\square}+ (m-q)\mathcal{N}=\square$
(see \cite{Tan}). Hence, the subgroups $H^q_{(0)}(M,E)\cong H^{0,q}_{(s)}(M)$, $q<m$, are all trivial for eigenvalues $s>0$.

(2) This result follows directly from Theorem~\ref{RegVan} and the shift (\ref{shift}).
\end{proof}

\begin{Example} Let us assume that the CR dimension $m\geq 2$ is even and the weight $f$ of $F\to N$ is a positive factor of $\frac{m}{2}+1$, i.e.,
$s:=-\frac{m+2}{2f}\in\ZZ$. Then $E^2=F^{2s}=\mathcal{K}$, i.e., $F^s\to M$ defines a spin structure
on the CR manifold~$M$. Hence, if the scalar curvature ${\rm scal}^h$ of the underlying K\"ahler manifold is positive,
the Kohn--Rossi subgroup $H_{(s)}^{0,\frac{m}{2}}(M)\cong H_{(0)}^{\frac{m}{2}}(M,F^s)$ is trivial.

In other words, in this situation the (untwisted) Kohn--Rossi subgroup \[ H_{(s)}^{0,\frac{m}{2}}(M) \neq \{0\}\] to the negative $\mathcal{N}$-eigenvalue $s=-\frac{m+2}{2f}\in\ZZ$
is an obstruction to positive Webster scalar curvature on $M$.
\end{Example}

\section[Examples of CR principal U(1)-bundles]{Examples of CR principal $\boldsymbol{{\rm U}(1)}$-bundles} \label{Sec12}

We discuss some examples of CR principal ${\rm U}(1)$-bundles with applications of vanishing theorems
for harmonic spinors and Kohn--Rossi cohomology. The construction here is based on the underlying K\"ahler manifold.

For our construction, let $\big(N^{2m},\omega,J\big)$, $m\geq 2$,
be a~closed K\"ahler manifold such that a~non-trivial multiple $\alpha:=c[\omega]$, $c\in\RR\setminus \{0\}$, of the K\"ahler class $[\omega]\neq 0$
is integral, i.e., $\alpha\in H^2(N,\ZZ)$. By {\em Lefschetz's theorem on $(1,1)$-classes}, there exists
some Hermitian line bundle $\mathcal{L}'\to N$ with first Chern class $c_1(\mathcal{L}')=\alpha$.
Let $\pi\colon \mathcal{M}\to N$ be the corresponding
principal ${\rm U}(1)$-bundle to which $\mathcal{L}'\to N$ is associated. By the {\em $dd^c$-lemma},
this bundle admits a connection form $A_\omega$ with curvature ${\rm d}A_\omega=\pi^*\omega$,
the lift of the K\"ahler form.
The lift of the complex structure $J$ to the horizontal distribution $H(\mathcal{M})$ of the connection~$A_\omega$ gives rise to a CR structure $(H(\mathcal{M}),J)$
on~$\mathcal{M}$, which is by construction strictly pseudoconvex of hypersurface type with CR dimension~$m$.

In addition, we set $\theta:=2A_\omega$. This is an adapted pseudo-Hermitian form, which has by construction vanishing torsion $\tau=0$,
the characteristic vector $T$ is regular and $g_\theta=\frac{1}{2}{\rm d}\theta(\cdot,J\cdot)$
is the lift of the K\"ahler metric to $H(\mathcal{M})$. Thus,
\[
(\mathcal{M}, H(\mathcal{M}), J)
\]
is a CR principal ${\rm U}(1)$-bundle with pseudo-Hermitian form $\theta$ over the K\"ahler manifold~$N$.
Let us call $\pi\colon \mathcal{M}\to N$ the {\em CR $\alpha$-bundle}.

\begin{Example}
Let $N^{2m}$ be a closed K\"ahler--Einstein manifold of positive scalar curvature ${\rm scal}^h>0$ and let
$\pi\colon \mathcal{M}\to N$ be the principal ${\rm U}(1)$-bundle to which the anticanonical bundle $\mathcal{K}'^{-1}\to N$ is associated.
The Ricci $2$-form is given by $\rho=\frac{{\rm scal}^h}{2m}\omega$ and the first Chern class is
$c_1(N)=-c_1(\mathcal{K}')=\big[\frac{1}{2\pi}\rho\big]\in H^2(N,\ZZ)$. Hence, $\pi\colon \mathcal{M}\to N$ is the
CR $c_1(N)$-bundle. The adapted pseudo-Hermitian form $\theta$ is some multiple of the Levi-Civita connection on~$\mathcal{M}$, and
the lift of the Ricci $2$-form $\rho$ along $\pi\colon \mathcal{M}\to N$ is the pseudo-Hermitian Ricci form $\rho_\theta$ of $\theta$ on~$\mathcal{M}$. This shows that $\rho_\theta$ is positive definite in this case.

Let $\mathcal{M}$ be given some ${\rm spin}^\CC$ structure of weight $\ell\in\ZZ$, determined by some line bundle $E\to \mathcal{M}$.
For $|\ell|<m+2$, Proposition~\ref{vani} shows that any harmonic spinor on $\mathcal{M}$ is a section of the extremal bundles.
Accordingly, Theorem~\ref{VanKR} shows that, for $|\ell|\leq m+2$ and any $1\leq q\leq m-1$, the $q$th Kohn--Rossi group
$H^q(\mathcal{M},E)=\{0\}$ is trivial. In particular, if~$N^{2m}$ ($m$~even) is spin, so is~$\mathcal{M}$ and
$H^{\frac{m}{2}}\big(\mathcal{M},\sqrt{\mathcal{K}}\big)=\{0\}$.

For example, the {\em Hopf fibration} $\pi\colon S^{2m+1}\to \CC P^{m}$ over the complex projective space is a CR principal ${\rm U}(1)$-bundle of this
kind. In fact, the associated line bundle to $\pi$ is $\mathcal{K}'^{-1}$, and $S^{2m+1}$ is equipped with the {\em standard CR structure}.
Hence, for any ${\rm spin}^\CC$ structure of weight $|\ell|\leq
m+2$ and $0<q<m$,
we have $H^q\big(S^{2m+1},E\big)=\{0\}$ (cf.~\cite{Tan}). (Note that~$\CC P^{m}$ is spin only for~$m$ odd. However, the standard CR manifold $S^{2m+1}$ is spin for any~$m$.)
\end{Example}

\begin{Example} Let $N^{2m}$, $m\geq 2$ and even, be some closed manifold with {\em hyperK\"ahler metric}~$h$, i.e.,
$h$ is Ricci-flat and its {\em holonomy group} is contained in ${\rm Sp}\big(\frac{m}{2}\big)$.
In particular, $N$ is spin and admits some parallel spinor $\phi_0$ in $\Sigma^0(N)$, i.e., $H^{\frac{m}{2}}\big(N,\mathcal{O}\big(\sqrt{\mathcal{K}'}\big)\big)\neq \{0\}$.

We assume now that $h$ is a {\em Hodge metric}, i.e., the corresponding K\"ahler class $[\omega]\in H^2(N,\ZZ)$ is integral.
E.g., any {\em projective K3 surface} admits some integral K\"ahler class~$[\omega]$.
Then the CR $[\omega]$-bundle $\pi\colon \mathcal{M}\to N$ exists. Furthermore, the pullback of the spin structure on $N$ gives rise to some spin structure
$\sqrt{\mathcal{K}}\to\mathcal{M}$ and the lift $\psi_0:=\pi^*\phi_0$ is a parallel section of $\Sigma^0(\mathcal{M}(H))$ with respect to
the spinorial Webster--Tanaka connection~$\nabla^\Sigma$.
In fact, the {\em basic holonomy group} of the adapted pseudo-Hermitian structure $\theta$ on $\mathcal{M}$
is contained in ${\rm Sp}\big(\frac{m}{2}\big)$ (cf.~\cite{LeiN2}). In
particular, $\psi_0$ is some harmonic spinor in $\Sigma^0(\mathcal{M}(H))$ and the Kohn--Rossi group
$H^{\frac{m}{2}}\big(\mathcal{M},\sqrt{\mathcal{K}}\big)\neq \{0\}$ is non-trivial. Accordingly, Corollary~\ref{ObsKR} states that the CR manifold $\mathcal{M}$ admits no pseudo-Hermitian
structure of positive Webster scalar curvature. Similarly, we can apply Corollary~\ref{cordem}, since $H^{\frac{m}{2}}\big(N,\mathcal{O}\big(\sqrt{\mathcal{K}'}\big)\big)\neq \{0\}$.

Observe that $\psi_0$ is in the kernel of the twistor operator~$\mathcal{P}_0$ as well, i.e.,
$\psi_0$ is an example for a CR twistor spinor of weight $\ell=0$ on some closed, strictly pseudoconvex CR manifold.

Finally, note that in the spin case there is a theory of {\em K\"ahlerian twistor spinors} on
the underlying~$N$ (see~\cite{Mor,Pil}). The lift of such twistor spinors gives rise
to projectable spinors in the kernel of twistor operators on $(M,\theta)$. For the case $\mu_q=\ell=0$, this gives rise to CR twistor spinors.
However, K\"ahlerian twistor spinors on $N$ with $\mu_q=0$ are necessarily parallel. Hence, as in our example,
the corresponding CR twistor spinors on~$M$ are parallel as well. Apart from the standard CR sphere, we do not know
further examples (of non-extremal weight) on closed CR manifolds.
\end{Example}

\subsection*{Acknowledgements}

I would like to thank the referees for their helpful comments.

\pdfbookmark[1]{References}{ref}
\LastPageEnding

\end{document}